%% file: ms.tex
\input{preamble/final}

\mathtoolsset{showonlyrefs}

\makeatletter
\@namedef{subjclassname@2020}{\textup{2020} Mathematics Subject Classification}
\makeatother

\title[Free energy analyticity of XY and Debye screening in the Coulomb gas]{Free energy analyticity of the disordered XY model and Debye screening in the 2D Coulomb gas}

\author{Lucas D'Alimonte}
\address{CNRS, Sorbonne Université, LPSM}
\email{dalimonte@lpsm.paris}

\author{Piet Lammers}
\address{CNRS, Sorbonne Université, LPSM}
\email{piet.lammers@cnrs.fr}

\date{27 March 2026}

\newcommand\leqstoch{\leq_{\operatorname{stoch}}}

\usepackage[
    style=alphabetic,
    date=year,
    maxnames=3,
    maxbibnames=99
]{biblatex}
\AtEveryBibitem{%
  \clearfield{url}%
  \clearfield{urldate}%
  \clearfield{urlday}%
  \clearfield{urlmonth}%
  \clearfield{urlyear}%
  \clearfield{doi}%
  \clearfield{isbn}%
  \clearfield{issn}%
}

\begin{document}
\maketitle

\begin{abstract}
We consider three models of statistical mechanics:
the classical XY model in arbitrary dimension, the lattice Coulomb gas in dimension two,
and the square well model in arbitrary dimension.
For each of these three models, we prove that the free energy is analytic in the disordered regime
(the square well model is disordered at any positive temperature).
In order to prove these results, we prove that the Gibbs measures
of these models are factors of i.i.d.\ with information clusters of exponentially decaying
size (volume).
In the case of the Coulomb gas, we obtain a strong version of Debye screening with an arbitrary number of arbitrary local observables of the Coulomb gas,
and we prove that the Debye phase contains the complement of the Berezinskii--Kosterlitz--Thouless phase of the Villain model.
\end{abstract}

\vspace*{1cm}
\renewcommand{\thepart}{\Roman{part}}
\setcounter{tocdepth}{1}
\tableofcontents

\newpage
\input{sections_new/intro/main.tex}

\part{The square well model}
\label{part:SWM}
\input{sections_new/SWM/main.tex}

\part{The XY model}
\label{part:XY}
\input{sections_new/xy/main.tex}

\input{sections_new/coulomb/main.tex}





\subsection*{Acknowledgements}
\addtocontents{toc}{\vspace{1em}}
\addcontentsline{toc}{section}{Acknowledgements}
The authors thank Diederik van Engelenburg, Nathan de Montgolfier, Sébastien Ott, and Avelio Sepúlveda for useful discussions and comments on the manuscript,
as well as Christophe Garban for suggesting to incorporate the square well model in the analysis.
This research was supported by the French National Research Agency (ANR), project number \texttt{ANR-23-CPJ1-0150-01}.
This research was supported in part by the International Centre for Theoretical Sciences (ICTS) for participating in the programme \emph{Geometric methods in Percolation and Spin Models 2026} (code: \texttt{ICTS/GPSM2026/03}).

\printbibliography
\end{document}

%% file: preamble/final.tex
\documentclass[dvipsnames,11pt,reqno,twoside,final]{amsart}
\usepackage[a4paper,left=3cm,right=3cm,top=3cm,bottom=3cm]{geometry}
\input{preamble/main}

%% file: preamble/main.tex
\usepackage[T1]{fontenc}
\usepackage[utf8]{inputenc}
\usepackage{xcolor}
\usepackage{amssymb,mathtools}  
\usepackage{dsfont}
\usepackage{enumitem}
\usepackage{subcaption}
\usepackage{booktabs}
\usepackage{algpseudocode}
\usepackage{centernot}
\input{preamble/displays}
\input{preamble/environments}
\input{preamble/typography}
\input{preamble/references}
\input{preamble/todos}

\input{preamble/symbols}

%% file: preamble/displays.tex
\usepackage{preamble/mhequ}

%% file: preamble/environments.tex
\usepackage{amsthm}

\newcounter{counterEnvDefault}
\numberwithin{counterEnvDefault}{section}

\theoremstyle{plain}

\newtheorem{lemma}[counterEnvDefault]{Lemma}
\newtheorem*{lemma*}{Lemma}

\newtheorem{theorem}[counterEnvDefault]{Theorem}
\newtheorem*{theorem*}{Theorem}

\newtheorem{proposition}[counterEnvDefault]{Proposition}
\newtheorem*{proposition*}{Proposition}

\newtheorem*{corollary*}{Corollary}

\newtheorem*{assumption*}{Assumption}

\newtheorem*{example*}{Example}
\theoremstyle{definition}

\newtheorem*{exercise*}{Exercise}

\newtheorem{definition}[counterEnvDefault]{Definition}
\newtheorem*{definition*}{Definition}

\newtheorem{remark}[counterEnvDefault]{Remark}
\newtheorem*{remark*}{Remark}

\newtheorem*{claim*}{Claim}

\newtheorem*{assertion*}{Assertion}

%% file: preamble/typography.tex
\usepackage{microtype}
\renewcommand\phi\varphi
\renewcommand\epsilon\varepsilon

%% file: preamble/references.tex
\usepackage{hyperref}
\definecolor{colorlinks}{RGB}{0, 24, 168}
\definecolor{colorcites}{RGB}{124, 10, 2}
\hypersetup{
    colorlinks=true,
    linkcolor=colorlinks,
    citecolor=colorcites,
    urlcolor=colorlinks,
    pdfborder={0 0 0}
}

%% file: preamble/todos.tex
\usepackage{xargs}
\usepackage[colorinlistoftodos,prependcaption,textsize=tiny]{todonotes}

\newcommandx\work[2][1=]{\todo[linecolor=RoyalBlue,backgroundcolor=RoyalBlue!25,bordercolor=RoyalBlue,#1]{\textsc{todo} #2}}
\newcommandx\comment[2][1=]{\todo[linecolor=OliveGreen,backgroundcolor=OliveGreen!25,bordercolor=OliveGreen,#1]{\textsc{comment} #2}}
\newcommandx\mistake[2][1=]{\todo[linecolor=red,backgroundcolor=red!25,bordercolor=red,#1]{\textsc{mistake} #2}}
\newcommandx\improve[2][1=]{\todo[linecolor=orange,backgroundcolor=orange!25,bordercolor=orange,#1]{\textsc{improve} #2}}
\newcommandx\change[2][1=]{\todo[linecolor=yellow,backgroundcolor=yellow!25,bordercolor=yellow,#1]{\textsc{change} #2}}
\newcommandx\mem[2][1=]{\todo[linecolor=orange,backgroundcolor=orange!25,bordercolor=orange,#1]{\textsc{mem} #2}}
\newcommandx\status[2][1=]{\todo[linecolor=Blue,backgroundcolor=Blue!25,bordercolor=Blue,#1]{\textsc{Status} #2}}

%% file: preamble/symbols.tex
\newcommand\blank{\,\cdot\,}

\newcommand{\eps}{\epsilon}
\newcommand{\om}{\omega}

\newcommand\XY{{\operatorname{XY}}}

\newcommand{\bfz}{\mathbf{0}}

\newcommand\ind[1]{\mathds{1}_{#1}}

\newcommand\diffi{{\,\mathrm{d}}}
\newcommand\diff{{\mathrm{d}}}
\newcommand\e{\mathrm{e}}

\newcommand\C{\mathbb C}

\newcommand\E{\mathbb E}

\newcommand\N{\mathbb N}

\renewcommand\P{\mathbb P}

\newcommand\R{\mathbb R}
\renewcommand\S{\mathbb S}

\newcommand\Z{\mathbb Z}

\newcommand\calC{\mathcal C}

\newcommand\calE{\mathcal E}
\newcommand\calF{\mathcal F}

\newcommand\calL{\mathcal L}
\newcommand\calM{\mathcal M}
\newcommand\calN{\mathcal N}

\newcommand\calT{\mathcal T}

%% file: sections_new/intro/main.tex
\input{sections_new/intro/01_preface.tex}

\input{sections_new/intro/02_results.tex}

\input{sections_new/intro/03_organisation.tex}

%% file: sections_new/intro/01_preface.tex
\section{Preface}

\subsection{Classifying phase transitions by regularity}
The idea of classifying phase transitions of lattice models according to the regularity of their free energy was introduced by Paul Ehrenfest in 1933~\cite{ehrenfest}.
He predicted that a generic thermodynamical quantity --- the \emph{free energy} --- considered as a function of the parameters of the system (often the \emph{temperature}) should be analytic, except at special \emph{transition points}.
At those points a \emph{phase transition} occurs,
the nature of which can be further investigated by studying the regularity of the free energy.
In Ehrenfest's terminology, the transition is said to be of \emph{$k$-th order} at some point $\beta_c$
if the free energy at that point is $C^{k-1}$ but not $C^k$. 
The determination of the order of phase transitions is of deep theoretical interest, and has been thoroughly discussed since the systematic study of statistical mechanics.

\subsection{Probabilistic perspectives}
Phase transitions can also be characterised in terms of qualitative changes in the probabilistic behaviour
of the model when the inverse temperature passes certain threshold values
(spontaneous magnetisation in the Ising model, the appearance of an infinite cluster in percolation theory, etc.).
The probabilistic viewpoint provides a welcome testing ground for mathematicians for the analysis of phase transitions,
and those probabilistic questions can be investigated using tools coming from several fields
(combinatorics, complex analysis, measure theory, algebra, and others).
This also gives rise to an integrated approach:
analysing the relation between the regularity of the free energy (or other thermodynamical quantities) and the probabilistic behaviour at and around criticality.

\subsection{The BKT transition}
Berezinskii~\cite{Berezinskii_1971_DestructionLongrangeOrder,Berezinskii_1972_DestructionLongrangeOrder} and later Kosterlitz and Thouless~\cite{KosterlitzThouless_1973_OrderingMetastabilityPhase} conjectured
a new type of phase transition in two-dimensional media with a continuous symmetry.
This phase transition is driven by the changing behaviour of so-called \emph{topological defects},
and a key feature of the prediction in~\cite{KosterlitzThouless_1973_OrderingMetastabilityPhase} is that the free energy
is smooth (but not analytic) at the critical point (\emph{infinite order} in Ehrenfest's terminology).
This transition is now called the Berezinskii--Kosterlitz--Thouless (BKT) transition.
While this prediction is widely expected to be correct,
a rigorous analysis of the near-critical window is still missing due to the lack of
precise tools that apply in the nonperturbative regime.
The 2D XY model is not expected to be integrable, and while progress is being made
on its probabilistic analysis, the lack of a hard quantative input makes it difficult to obtain a complete picture at the moment.

Let us compare this to recent developments in the context of the 2D random-cluster model
with cluster weight $q\geq 1$, and the associated six-vertex model.
By using the Bethe Ansatz (a manifestation of integrability),
it has been proved that the phase transition is continuous for $q\in[1,4]$~\cite{DCSidoraviciusTassionContinuity}
and discontinuous for $q>4$~\cite{DCgagnebinHarelManolescuDiscontinuity}.
This is closely related to the behaviour of the free energy
of the model as a function of $q$ and of the slope in the six-vertex model.
We now understand the behaviour for small slopes for all $q>0$~\cite{Duminil-CopinKozlowskiKrachun_2022_SixVertexModelsFree}.
This implies that the \emph{mass} of the model (the inverse of the correlation length, which is zero for $q\leq 4$)
is analytic on $(4,\infty)$ and
drops smoothly to zero near $q=4$.
Similarly, the explicit expressions in~\cite{Duminil-CopinKozlowskiKrachun_2022_SixVertexModelsFree} reveal that the free energy of the model is smooth in $q$ and analytic everywhere except at $q=4$.
These behaviours are reminiscent of the BKT transition, but
we cannot at the moment answer these questions for the XY model
with similar methods because exact integrability (leading to explicit expressions for the free energy) appears not to be available.

\subsection{Analyticity away from criticality}
In the absence of integrability, this article aims to progress on the probabilistic analysis of the XY model
in the \emph{subcritical} regime (in the open interval $[0,\beta_c)$).
It derives analyticity of the free energy in this interval.
As a by-product, we also obtain analyticity of the free energy
of the closely related 2D lattice Coulomb gas on the same subcritical
interval,
and we establish that the interactions between the Coulomb particles
decay exponentially fast in the distance in the same temperature regime
(Debye screening).

While the question of analyticity throughout the entire subcritical regime may appear simple at first sight, it is not uncommon that subtle complications arise.
In the tradition of deriving analyticity,
let us first mention the remarkable and celebrated work of Lee and Yang~\cite{Lee_yang_1,Lee_yang_2} regarding the analyticity of the free energy as a function of the magnetic field, which is the first successful non-perturbative approach of this type of question.
Georgakopoulos and Panagiotis proved in~\cite{PG_analyticity} that the density of the infinite cluster
in Bernoulli percolation is analytic in the whole supercritical regime,
providing a first example of a non-perturbative proof
of analyticity in a percolation model using some form of renormalisation or coarse-graining argument combined with involved combinatorial constructions.
We also mention the work~\cite{panagiotis_severo_analyticity_GFF}, that establishes a similar result in the case of a related model, the \emph{Gaussian free field percolation}.  
In the Ising model, the most complete picture has been obtained by Ott, who proved in~\cite{Ott_Analyticity_Ising,Ott_Analyticity_RCM} that the free energy of the Ising model is analytic in $\beta$ everywhere but at $\beta_c$.
In perturbative setting, the most complete picture was described in the context of general Gibbs measures in~\cite{analyticity_pressure_general}.

Finally, we mention that the correspondence between Ehrenfest's notion of phase transition and the probabilistic viewpoint was questioned by Griffiths, who discovered in~\cite{griffiths_sing} the very intriguing phenomenon that now bears the name of \emph{Griffiths singularities}. In the setting of \emph{diluted magnets} (Ising models with random Hamiltonians), he showed that the analyticity of the free energy of the model could be broken \emph{away of the critical point}, due to the random fluctuations of the Hamiltonian. 
In the absence of those random defects (on so-called \emph{pure} models), this phenomenon should however not exist, and Ehrenfest's correspondance is expected to be exact. 

\subsection{Background on the XY model and the discrete 2D Coulomb gas}

The critical behaviour of the XY model depends starkly on the dimension of the underlying lattice.
In dimension two, it has been known since the 1960s that the XY model does not undergo a magnetisation transition (like the Ising model) due to the celebrated work of Mermin and Wagner~\cite{mermin_wagner}.
At the same time, numerical simulations opened up the idea that the model might still undergo a phase transition
of a different type~\cite{StanleyKaplan_1966_PossibilityPhaseTransition}.
As mentioned before, such a phase transition was theoretically motivated by Berezinskii~\cite{Berezinskii_1971_DestructionLongrangeOrder,Berezinskii_1972_DestructionLongrangeOrder} and Kosterlitz and Thouless~\cite{KosterlitzThouless_1973_OrderingMetastabilityPhase}.
At and below the critical temperature, it is physically expected that the model
has a conformally invariant scaling limit~\cite{Ginsparg_1988_Curiosities1}.

While a clear picture of the near-critical behaviour of the 2D XY model remains elusive to mathematicians,
the model has been the subject of intense research since the 1980s.
In their landmark work, Fröhlich and Spencer~\cite{FrohlichSpencer_1981_KosterlitzThoulessTransitionTwodimensional} (cf.~\cite{KharashPeled_2017_FrohlichSpencerProofBerezinskiiKosterlitzThouless}) proved that the low-temperature XY model
exhibits polynomial decay of correlations, thus rigorously proving that a phase transition occurs
(exponential decay at high temperature is classical).
The authors obtain this result by viewing the dual of the XY model (an \emph{integer-valued height function})
as a perturbation around the \emph{discrete Gaussian free field} (a \emph{real-valued height function}), and by using a renormalisation group analysis to control the effect of the perturbation.
The argument is robust; for example, it was shown that the argument is stable under certain perturbations~\cite{GarbanSepulveda_2023_StatisticalReconstructionGFF},
and that it can be extended to a larger range of potentials, even at a slope~\cite{OttSchweiger_2025_QuantitativeDelocalizationSolidonsolid}.
On the height function side, it was recently proved (for a slightly different potential)
that the height function converges to the Gaussian free field in the scaling limit at high temperature~\cite{BauerschmidtParkRodriguez_2024_DiscreteGaussianModel,BauerschmidtParkRodriguez_2024_DiscreteGaussianModela}.
We finally refer to~\cite{GarbanSepulveda_2023_QuantitativeBoundsVortex} for an analysis of the relations
between the Villain (modified XY) model, the Coulomb gas,
and the Gaussian free field,
and the way that correlations in the different model relate to one another.

Later, a second argument for a phase transition appeared,
based on the non-coexistence theorem in percolation (see~\cite[Theorem 9.3.1]{Sheffield_2005_RandomSurfaces}
and~\cite[Theorem 1.5]{Duminil-CopinRaoufiTassion_2019_SharpPhaseTransition}).
In~\cite{Lammers_2022_HeightFunctionDelocalisation} these ideas were used to show that
2D integer-valued height functions delocalise at high temperature,
and~\cite{EngelenburgLis_2023_ElementaryProofPhase,AizenmanHarelPeled_2022_DepinningIntegerrestrictedGaussian}
derives from this the polynomial decay of correlations in the low-temperature 2D XY model.
Further results were obtained in~\cite{Lammers_2023_DichotomyTheoryHeight,Lammers_2023_BijectingBKTTransition},
where it was proved that the BKT transition in the XY model transition is equivalent to the
delocalisation transition in the dual height function.
The duality transform was further explored in~\cite{EngelenburgLis_2025_DualityHeightFunctions}
to obtain GFF-type estimates and monotonicity in parameters in large generality.

The $\Z^2$ Coulomb gas alluded to above has a continuum counterpart, in which the Coulomb particles
live in $\R^2$ rather than $\Z^2$.
This complicates certain aspects of the model (for example,
it is problematic when particles of opposing charge come very close to each other).
Boursier and Serfaty~\cite{BoursierSerfaty_2024_DipoleFormationTwoComponent,BoursierSerfaty_2025_MultipoleBerezinskiiKosterlitzThoulessTransitions} recently obtained improved
electrostatic estimates for the continuum Coulomb gas (cf.~\cite{FrohlichSpencer_1981_KosterlitzThoulessTransitionTwodimensional}),
which allows for the derivation of a sequence of multipole transitions at temperatures accumulating
above the BKT temperature.
These multipoles are connected to transitions in the coefficients of a Mayer series arising in the free energy expansion. 

In dimension $d\geq 3$, the XY model is known to exhibit spontaneous magnetisation (or continuous symmetry breaking) at low temperature~\cite{FrohlichSimonSpencer_1976_InfraredBoundsPhase}
(cf.~\cite{GarbanSpencer_2022_ContinuousSymmetryBreaking} for a new proof based on~\cite{AbbeMassoulieMontanari_2018_GroupSynchronizationGrids}).
Moreover, it is known that in dimension $d\geq 5$, the model exhibits trivial (Gaussian) behaviour at the critical point~\cite{Frohlich_1982_TrivialityLphd4Theories}.
For the Ising model, such methods were recently extended to cover the dimension $d=4$~\cite{AizenmanDuminil-Copin_2021_MarginalTrivialityScaling}
(cf.~\cite{Aizenman_1982_GeometricAnalysisF4}).
Our mathematical knowledge on the higher dimensional XY model remains limited,
with key problems (such as continuity and sharpness of the phase transition) remaining open
in any dimension $d\geq 3$.

\subsection{Background on the square well model}

Finally, we apply our techniques to the square well model.
This model has real-valued spins and a quadratic interaction term, but the values
are conditioned to lie in the interval $[-1,1]$.
It is known that this conditioning destroys long-range interactions;
the two-point function exhibits exponential decay of correlations at all temperatures thanks
to the work of McBryan and Spencer~\cite{McBryanSpencer_1977_DecayCorrelationsInSOnsymmetric}.
Furthermore, the study of this model was pursued in~\cite{frohlich_random_surfaces} in which precise asymptotics of the \emph{mass} of the model are derived when the cutoff $[\pm 1]$ is replaced by $[\pm L]$ and $L$ tends to infinity.
The model is included in the analysis as it showcases the novel aspects of our argument
in a slightly simpler setting than the one of the XY model.
We also include a new proof of the McBryan--Spencer result of exponential decay of this model,
relying on the Fortuin--Kasteleyn--Ginibre inequality~\cite{FortuinKasteleynGinibre_1971_CorrelationInequalitiesPartially} and ``absolute-value-FKG inequality''~\cite{LammersOtt_2024_DelocalisationAbsolutevalueFKGSolidonsolid}.

%% file: sections_new/intro/02_results.tex
\section{Main results}

\newcommand\SWM{\operatorname{SWM}}
\newcommand\XYM{\operatorname{XY}}
\newcommand\Coulomb{\operatorname{Coulomb}}
\newcommand\Villain{\operatorname{Villain}}
\newcommand\GFF{\operatorname{GFF}}
\newcommand\Dirichlet{{\operatorname{Dirichlet}}}
\newcommand\BKT{{\operatorname{BKT}}}
\newcommand\free{{\operatorname{free}}}
\newcommand\ivgff{{\operatorname{IV-GFF}}}

For any $d\in\Z_{\geq 1}$, 
let $\Z^d$ denote the $d$-dimensional square lattice graph endowed with nearest-neighbour connectivity (denoted $\sim$).
A \emph{domain} is a finite subset of $\Z^d$.
Let $\Lambda_n:=(-n,n)^d\cap\Z^d$.

\subsection{Free energy of the XY model}

Let $\S^1\subset\C$ denote the unit circle in the complex plane.

\begin{definition}[XY model]
  \label{def:XY}
    Fix $d\in\Z_{\geq 1}$ and $\beta\in\R_{\geq 0}$.
    For any domain $\Lambda$ and $\zeta:\Z^d\to\S^1$,
    the \emph{XY model} on $\Lambda$ with boundary condition $\zeta$
    at inverse temperature $\beta$ is the probability measure  on $\sigma\in(\S^1)^\Lambda$ given by
    \[
        \diff\mu_{\XYM,\Lambda,\beta}^\zeta(\sigma) = \frac1{Z_{\XYM,\Lambda,\beta}^\zeta} e^{-\beta H_{\XYM,\Lambda}^\zeta(\sigma)} \diff\sigma,
    \]
    where $Z_{\XYM,\Lambda,\beta}^\zeta$ is the normalising constant,
    $\diff\sigma$ the Haar measure, and
    \[
        H_{\XYM,\Lambda}^\zeta(\sigma) := \sum_{\substack{\{x,y\}\subset\Lambda\\x\sim y}} |\sigma_x-\sigma_y|^2/2 + \sum_{\substack{x\in\Lambda,\,y\in\Lambda^c\\x\sim y}} |\sigma_x-\zeta_y|^2/2.
    \]
    The \emph{critical inverse temperature} is defined as the largest $\beta_c=\beta_c(d)\in(0,\infty)$
    such that $\mu_{\XY,\Lambda_n,\beta}^{+1}[\sigma_0]$ decays exponentially fast in $n$
    for any fixed $\beta\in[0,\beta_c)$.
    Classical arguments based on correlations inequalities imply the existence of an infinite-volume Gibbs measure with $+1$ boundary conditions at inverse temperature $\beta\geq 0$, that we refer to as $\mu^{+1}_{\XY, \beta}$.
    The label $\XYM$ is omitted from notations when it is clear from the context.
\end{definition}

\begin{theorem}[Analyticity of the free energy of the XY model]
    \label{thm:analyticity_XY}
    Fix $d\in\Z_{\geq 1}$.
    Then
    \[
        \mathbf{\mathfrak{f}}_{\XY} :\R_{\geq 0}\to\R,\,\beta\mapsto \lim_{n\to\infty} \frac1{|\Lambda_n|} \log Z_{\XYM,\Lambda_n,\beta}^{+1},
    \]
    the free energy of the model, is analytic on $[0,\beta_c(d))$.
\end{theorem}

\subsection{Free energy of the discrete two-dimensional Coulomb gas}

Now consider $d=2$.
Recall that $\Lambda_n=(-n,n)^2\cap\Z^2$ is the base graph for the XY model.
We also use it as the base graph for the Villain model, which is a slight modification of the XY model
(described formally in Section~\ref{sec:coulomb}).
The Coulomb gas is defined on the faces $F(\Lambda_n)$ of $\Lambda_n$,
that is, the unit squares with vertices in $[-n,n]^2\cap\Z^2$.
We think of $F(\Lambda_n)$ as a graph; two faces are adjacent if they share an edge
(faces in the bulk have four neighbors; edges on the boundary two or three).
We let $\Delta$ denote the Laplacian on this graph of faces, and $\Delta^{-1}$ its inverse.

\begin{definition}[Two-dimensional Coulomb gas]
  \label{def:Coulomb}
    Fix $d=2$ and $\beta\in\R_{\geq 0}$.
    The \emph{two-dimensional lattice Coulomb gas} on $\Lambda_n$ with free boundary conditions at inverse temperature $\beta$
    is the probability measure on $q\in\Z^{F(\Lambda_n)}$
    given by
    \begin{equation}
      \label{eq:Coulomb_gas_definition}
        \mu_{\Coulomb,\Lambda_n,\beta}^\free[\{q=Q\}] = \frac{1}{Z_{\Coulomb,\Lambda_n,\beta}^\free} \mathbf{1}\Big[\sum\nolimits_{f\in F(\Lambda_n)}Q(f)=0\Big]\cdot e^{-\frac{\beta(2\pi)^2}{2} \langle Q , -\Delta^{-1} Q \rangle}.
    \end{equation}
    Notice that, contrary to the XY model,
    the Hamiltonian of the Coulomb gas is non-local due to the presence of the inverse Laplacian.
\end{definition}

Let us now briefly describe the relation of this model with the Villain model.
The Villain model is an XY model on $\Lambda_n$ with a modified Hamiltonian.
It is closely related to the Coulomb gas, as it may be viewed as the product measure of the Coulomb gas and a Gaussian free field (GFF) on $\Lambda_n$.
This relation was described in~\cite{GarbanSepulveda_2023_QuantitativeBoundsVortex}.
Direct consequences are:
\begin{itemize}
    \item For any fixed $\beta$, we have
    \begin{equation}
        \label{eq:Villain_Coulomb_relation}
        Z_{\Villain,\Lambda_n,\beta}^{+1} = Z_{\Coulomb,\Lambda_n,\beta}^\free \cdot Z_{\GFF,\Lambda_n,\beta}^0,
    \end{equation}
    see \cite[Page~672, Remark~10]{GarbanSepulveda_2023_QuantitativeBoundsVortex},
    \item For any fixed $\beta$,
    one can take a sample from $\mu_{\Coulomb,\Lambda_n,\beta}^\free$
    via the following procedure (detailed in~\cite[Section~4]{GarbanSepulveda_2023_QuantitativeBoundsVortex}):
    \begin{enumerate}
        \item Sample $\sigma\sim\mu_{\Villain,\Lambda_n,\beta}^{+1}$,
        \item Sample additional i.i.d.\ random variables $\zeta$ on the edges of $\Lambda_n$,
        \item The charge $q_f$ at any face $f\in F(\Lambda_n)$ is now a function of the spins $\sigma_u$
        on the four corners $u$ of $f$ and $\zeta_e$ on the four edges $e$ of $f$.
    \end{enumerate}
    In particular, the Coulomb gas can be sampled locally in terms of the Villain model.
\end{itemize}
We stress that the relation between the Coulomb gas and the Villain model is \emph{not} a Fourier type transform:
there is no temperature inversion. Rather, the Coulomb gas is defined on the dual graph simply because it encodes topological features along the faces of the primal graph on which the Villain model is defined.

Like the XY model, the Villain model in 2D has a special inverse temperature $\beta_{\BKT}$
such that correlations decay exponentially fast for $\beta<\beta_{\BKT}(\Villain)$ and polynomially fast for $\beta\geq\beta_{\BKT}(\Villain)$.
As discussed in~\cite{Berezinskii_1972_DestructionLongrangeOrder,FrohlichSpencer_1981_KosterlitzThoulessTransitionTwodimensional,GarbanSepulveda_2023_QuantitativeBoundsVortex},
this phase transition is driven by qualitative changes in the Coulomb gas measures $\mu_{\Coulomb,\Lambda_n,\beta}^\free$.
This must necessarily be the case as the Gaussian free field does not undergo any phase transition;
changing $\beta$ simply scales the field by a constant factor.
As Theorem~\ref{thm:analyticity_XY} extends directly to the Villain model,
we obtain the following results for the Coulomb gas.

\begin{theorem}
  \label{thm:Coulomb_analytic}
  Fix $d=2$.
  Then the free energy of the Coulomb gas
  \[
    \mathbf{\mathfrak{f}}_{\mathrm{Coulomb}} :\R_{\geq 0}\to\R,\,\beta\mapsto \lim_{n\to\infty} \frac1{|\Lambda_n|} \log Z_{\Coulomb,\Lambda_n,\beta}^\free,
   \]
   is analytic on $[0,\beta_{\BKT}(\Villain))$.

   The same holds true if we replace free boundary conditions by Dirichlet boundary conditions,
   see~\cite{GarbanSepulveda_2023_QuantitativeBoundsVortex} for details.
\end{theorem}

This theorem follows directly from Equation~\eqref{eq:Villain_Coulomb_relation}, analyticity of the free energy of the Villain model on
$[0,\beta_{\BKT}(\Villain))$,
and analyticity of the free energy of the Gaussian free field on $[0,\infty)$.
Changing the boundary conditions (Dirichlet vs.\ free) modifies the partition function of $\log Z_{\Coulomb,\Lambda_n,\beta}$ by $O(|\partial\Lambda_n|)=o(|\Lambda_n|)$, as can be observed from the relation~\eqref{eq:Villain_Coulomb_relation}, and therefore the theorem extends to Dirichlet boundary conditions.

\subsection{Free energy of the integer-valued Gaussian free field (IV-GFF)}

For this discussion and this result, we also stay in two dimensions, that is we set the base graph to be $\Z^2$ endowed with its nearest-neighbour structure. 
The IV-GFF is a model of \emph{height function} on the vertices of $\Z^2$, known to be the dual height function of the Villain model. 

\begin{definition}[$\ivgff$]
\label{def:ivgff}
Fix $d=2$ and $\beta \in \R_{>0}$. For any domain $\Lambda$ and any $\zeta: \Z^2 \mapsto \Z$, the \emph{two-dimensional lattice integer-valued Gaussian free field} (2D IV-GFF) --- also known as \emph{discrete Gaussian model} --- on $\Lambda$ with boundary conditions $\zeta$ at inverse temperature $\beta$ is the probability measure on $\phi: \Z^2 \mapsto \Z$ given by 
\[ \mu^\zeta_{\ivgff, \Lambda, \beta} [\phi =f ]  = \frac{1}{Z_{\ivgff, \Lambda, \beta}^\zeta} \e^{- \beta H^\zeta_{\ivgff, \Lambda}(f)}, \]
where $Z^\zeta_{\ivgff, \Lambda, \beta}$ is the normalising constant and 
\[ H^\zeta_{\ivgff, \Lambda}(f) = \sum_{\substack{\{x,y\}\subset \Lambda \\ x \sim y}}(f_x - f_y)^2 + \sum_{\substack{x \in \Lambda, y \in \Lambda^c \\ x \sim y}}(f_x - \zeta_y)^2.  \]
\end{definition}

\begin{theorem}
  \label{thm:ivgff_analytic}
  Fix $d=2$.
  Then the free energy of the $\ivgff$
  \[
    \mathbf{\mathfrak{f}}_{\ivgff} :\R_{\geq 0}\to\R,\,\beta\mapsto \lim_{n\to\infty} \frac1{|\Lambda_n|} \log Z_{\ivgff,\Lambda_n,\beta^{-1}}^0,
   \]
   is analytic on $[0,\beta_{\BKT}(\Villain))$.

   The same holds true if we replace 0 boundary conditions by free boundary conditions.
\end{theorem}

We invite the reader to pay attention to the unusual $\tfrac{1}{\beta}$, which is a consequence of a \emph{duality transform}. 
Indeed, it is well known that the partition function of the $\ivgff$ is related to that of the Coulomb gas and of the GFF by the following relation (see~\cite{GarbanSepulveda_2023_QuantitativeBoundsVortex}):
\[ Z^0_{\ivgff, \Lambda_n, \beta^{-1}} = Z^0_{\GFF, \Lambda_n, \beta^{-1}}Z^{\mathrm{free}}_{\mathrm{Coulomb}, \Lambda_n, \beta}. \]

As such, Theorem~\ref{thm:ivgff_analytic} is a direct consequence of Theorem~\ref{thm:Coulomb_analytic}.

\begin{remark}[Localised phase of the $\ivgff$]
Define the \emph{critical inverse temperature} of the $\ivgff$ as the maximal $\beta \geq 0$ such that the field \emph{delocalises},
that is to say that $\mu^0_{\ivgff, \Lambda_n, \beta}[|\phi_0|^2] \to \infty$ when $n \to \infty$. 
It was proved in~\cite{Lammers_2023_BijectingBKTTransition} that the critical inverse temperature of the $\ivgff$ coincides the inverse of that of the Villain model. 
Thus, Theorem~\ref{thm:ivgff_analytic} proves the analyticity of the free energy of the $\ivgff$ \emph{on its whole localised phase}. 
As in the case of the 2D Coulomb gas, this is a non-trivial consequence of our work, 
because the methods employed to prove analyticity of the free energy of the Villain and XY models would fail 
if applied directly to the $\ivgff$. 
For instance, its spins are unbounded and it admits several Gibbs measures, which are two major limitations for our method to apply directly. 
\end{remark}

\subsection{Free energy of the square well model (SWM)}

\begin{definition}[Square well model (SWM)]
  \label{def:swm}
    Fix $d\in\Z_{\geq 1}$ and $\beta\in\R_{\geq 0}$.
    For any domain $\Lambda$ and $\zeta:\Z^d\to[-1,1]$,
    the \emph{square well model} (SWM) on $\Lambda$ with boundary condition $\zeta$
    at inverse temperature $\beta$ is the probability measure  on $\alpha\in[-1,1]^\Lambda$ given by
    \[
        \diff\mu_{\SWM,\Lambda,\beta}^\zeta(\alpha) = \frac1{Z_{\SWM,\Lambda,\beta}^\zeta} e^{-\beta H_{\SWM,\Lambda}^\zeta(\alpha)} \diff\alpha,
    \]
    where $Z_{\Lambda,\beta}^\zeta$ is the normalising constant,
    $\diff\alpha$ the Lebesgue measure, and
    \[
        H_{\SWM,\Lambda}^\zeta(\alpha) := \sum_{\substack{\{x,y\}\subset\Lambda\\x\sim y}} (\alpha_x-\alpha_y)^2 + \sum_{\substack{x\in\Lambda,\,y\in\Lambda^c\\x\sim y}} (\alpha_x-\zeta_y)^2.
    \]
    The label $\SWM$ is omitted when no confusion is likely to arise.
\end{definition}

McBryan and Spencer proved in~\cite{McBryanSpencer_1977_DecayCorrelationsInSOnsymmetric} that this model exhibits exponential decay of correlations
for any fixed $d$ and $\beta$.
 Consequently there exists a unique infinite-volume Gibbs measure, referred to as $\mu_{\SWM, \beta}$
We derive here the following result.

\begin{theorem}[Analyticity of the free energy of the SWM]
    \label{thm:analyticity_SWM}
    Fix $d\in\Z_{\geq 1}$.
    Then
    \[
       \mathbf{\mathfrak{f}}_{\SWM} :\R_{\geq 0}\to\R,\,\beta\mapsto \lim_{n\to\infty} \frac1{|\Lambda_n|} \log Z_{\SWM,\Lambda_n,\beta}^0,
    \]
    the free energy of the model, is analytic.
\end{theorem}

\subsection{The three models are exponential factors of i.i.d.\ (EFIIDs)}

The definition of an exponential factor of i.i.d.\ is already present in the literature in spirit,
but not yet defined as a distinct mathematical object.
Below, we give a precise definition.
We start with a standard definition of (classical) factors of i.i.d.

\begin{definition}[Factor of i.i.d.]
    \label{def:FIID}
    Consider a locally finite transitive graph $G=(V,E)$
    (with respect to some group $\Theta\subset\operatorname{Aut}(G)$).
    Let $S$ denote a measurable space,
    and let $\mu$ denote the distribution of a $\Theta$-stationary random field on $S^V$.
    Let $(T,\lambda)$ denote a probability space, and consider $\P:=\lambda^V$.
    A \emph{factor of i.i.d.}\ (FIID) is a map $\phi:T^V\to S^V$
    with the following properties:
    \begin{itemize}
        \item $\phi$ is measurable with respect to the product $\sigma$-algebras on $T^V$ and $S^V$,
        \item $\phi$ is equivariant with respect to $\Theta$,
        \item If $X\sim\P$, then $\phi(X)\sim\mu$.
    \end{itemize}
\end{definition}

FIIDs are useful for encoding correlations in random fields.
For example, if $G=\Z^d$ and if $\phi(X)_u$
is measurable with respect to $X|_{u+\Lambda_n}$,
then $\mu$ is $2n$-dependent (w.r.t.~$\ell^\infty$).
Our models do not fall into this category,
as they exhibit long-range correlations (albeit very weak ones).
To encode those correlations, we first introduce the notion of \emph{local sets}.

\begin{definition}[Local set]
  \label{def:local_set}
  Let $V$ denote a vertex set and $T$ a measurable space.
  Consider the measurable space $T^V$ endowed with the product $\sigma$-algebra.
  A typical element is denoted $X\in T^V$.
  A \emph{local set} is a random subset $\calL=\calL(X)\subset V$
  such that for any $L\subset V$,
  the event $\{\calL=L\}$ is measurable with respect to $X|_L$.
\end{definition}

\begin{definition}[Exponential factor of i.i.d. (EFIID)]
  \label{def:EFIID}
  Consider the setting of Definition~\ref{def:FIID}.
   A FIID is called an \emph{exponential factor of i.i.d.}\ (EFIID) with tails $(c,C)$ for constants $c,C\in(0,\infty)$ if for any $u\in V$,
   there exists a connected local set $\calL_u=\calL_u(X)\ni u$ such that:
   \begin{itemize}
    \item The value of $\phi(X)_u$ is measurable with respect to $X|_{\calL_u}$,
    \item We have $\P[\{|\calL_u|\geq n\}]\leq C e^{-cn}$ for any $n\in\Z_{\geq 0}$.
   \end{itemize}
\end{definition}

Most work in this article is devoted to proving the following result.

\begin{theorem}[The models are EFIIDs]
    \label{thm:EFIID_SWM_XY}
    All three models alluded to above are EFIIDs.
    More precisely, we obtain the following results.
    \begin{itemize}
        \item[\emph{\textbf{XY}}] Consider the XY model in dimension $d\in\Z_{\geq 2}$
        at inverse temperature $\beta\in[0,\beta_c(d))$.
        Then the infinite-volume Gibbs measure $\mu_{\XY,\beta}^{+1}$
        is an EFIID.
        \item[\emph{\textbf{SWM}}] Consider the SWM in dimension $d\in\Z_{\geq 2}$
        at inverse temperature $\beta\in\R_{\geq 0}$.
        Then the infinite-volume Gibbs measure $\mu_{\SWM,\beta}^0$
        is an EFIID.
        \item[\emph{\textbf{Coulomb}}] Consider the two-dimensional Coulomb gas at $\beta\in[0,\beta_{\BKT}(\Villain))$.
        Then the sequence $(\mu_{\Coulomb,\Lambda_n,\beta}^\free)_n$
        tends to some infinite-volume measure $\mu_{\Coulomb,\beta}^\free$ on $\Z^{F(\Z^2)}$ as $n\to\infty$ in the local convergence topology,
        and this limit is an EFIID.
    \end{itemize}
    In the case of the Coulomb gas, we are careful not to call this measure a Gibbs measure, as the Hamiltonian of the model is non-local and therefore it is not so clear what it means exactly to be a Gibbs measure.

    In particular, these results imply that for any of the three models,
    correlations between bounded observables decay exponentially fast in the distance at which they are measured.
\end{theorem}

\begin{remark}[Debye screening in the Coulomb gas]
    The Hamiltonian in the definition of the Coulomb gas is long-range.
    However, it was predicted in the physics literature that charges of opposing signs should gather in such a way
    that these long-range effects cancel out over
    long distances, for sufficiently small $\beta$.
    This effect is called \emph{Debye screening}.
    Debye screening at perturbatively small $\beta$ were previously obtained in \cite{Brydges_1978_RigorousApproachDebye,Yang_1987_DebyeScreeningTwodimensional},
    and~\cite{Sepulveda_2022_ScalingLimitDiscrete} describes how
    exponential decay in the dual height function model implies Debye screening in the Coulomb gas.

    The current work, in combination with~\cite{Lammers_2023_DichotomyTheoryHeight,Lammers_2023_BijectingBKTTransition} sharpens the picture:
    the Debye phase in the Coulomb gas contains the complement of the BKT phase in the 
    Villain model and in the localised (exponential decay) phase in the dual height function model.
    Moreover, the EFIID structure implies screening in a strong sense, namely that any finite number of local observables of any type, mix exponentially fast in the distance at which they are measured. 
\end{remark}

EFIIDs are useful for studying analyticity~\cite{Ott_Analyticity_Ising,Ott_Analyticity_RCM}.
In fact, Theorem~\ref{thm:EFIID_SWM_XY}
essentially implies Theorems~\ref{thm:analyticity_XY}, \ref{thm:Coulomb_analytic}, and \ref{thm:analyticity_SWM}
(see the next section for details;
the situation is slightly more complicated in the case of Theorem~\ref{thm:Coulomb_analytic}).
The main contribution of this article is therefore to show that the three models are indeed EFIIDs.

%% file: sections_new/intro/03_organisation.tex
\section{Proof overview and article organisation}

Subsection~\ref{subsec:organisation_proof_EFIID} describes how the EFIIDs (Theorem~\ref{thm:EFIID_SWM_XY}) are constructed in the remaining parts of this article.
The existence of such EFIIDs forms the core result.
Subsection~\ref{subsec:conditional_proofs_2} describes how this result implies analyticity.
This implication was already in~\cite{Ott_Analyticity_Ising,Ott_Analyticity_RCM},
and we merely include a proof sketch.

\subsection{Organisation of the proof of Theorem~\ref{thm:EFIID_SWM_XY}}
\label{subsec:organisation_proof_EFIID}

We execute our construction separately for the SWM and for the XY model.
The former case is simpler, but already illustrates some of the key ideas.
It is contained in Part~\ref{part:SWM}.
The latter case is a bit more delicate (additional complications arise),
and the construction is contained in Part~\ref{part:XY}.

Both parts are organised in tandem.
\begin{itemize}
    \item\textbf{A new Glauber dynamic (Sections~\ref{sec:glauber} and~\ref{sec:glauber_XY}).} At the first step, we observe that the models enjoy certain monotonicity properties,
    and we construct a Glauber dynamic that is compatible with these monotonicity properties.
    The ultimate goal is to perform a coupling-from-the-past procedure.
    This is more complicated for our models than for models with discrete spins (such as the Ising model),
    since the natural monotone Glauber dynamic does not couple in finite time.
    We introduce an alternative Glauber dynamic that \emph{does} couple in finite time, \emph{and} has the correct monotonicity property.
    The construction of this alternative Glauber dynamic is one of the key ideas that makes the construction work.
    \item\textbf{Spatial mixing (Sections~\ref{sec:spatial_mixing_SWM} and~\ref{sec:XY_spatial_mixing}).}
    Then, we prove for each model that the finite-volume Gibbs measures mix exponentially fast.
    The proofs are model-dependent and do not involve the Glauber dynamic.  
    \item\textbf{Space-time mixing (Sections~\ref{sec:spacetime} and~\ref{sec:spacetime_XY}).}
    Then, we prove space-time mixing for this Glauber dynamic.
    This relies on the ideas in the work of Harel and Spinka~\cite{harel_spinka},
    but the proof is more complicated because the spin space above each vertex is continuous rather than finite.
    \item \textbf{Coarse-graining towards a subcritical percolation (Sections~\ref{sec:efiid_SWM} and~\ref{sec:efiid_XY}).}
    The space-time mixing property indicates that for fixed $\beta$,
    we may fix, once and for all, a space-time scale on which the Glauber dynamic couples extremely fast.
    More precisely, the coarse-grained dependency clusters in space-time are dominated
    by a subcritical Bernoulli percolation.
    In particular, those clusters exhibit exponential decay in their volume.
    This property remains true under projection (of the space-time clusters onto the spatial dimension),
    which yields the desired EFIID structure.
\end{itemize}
This proves Theorem~\ref{thm:EFIID_SWM_XY} for the XY model and the SWM,
but not for the Coulomb gas.
That requires two more steps.
\begin{enumerate}
    \item The Villain model has an EFIID exactly like the XY model (Theorem~\ref{thm:VILLAIN_EFIID}).
    \item The Coulomb gas can be locally sampled in terms of the Villain model;
    see our discussion following Definition~\ref{def:Coulomb}.
    This implies immediately that the EFIID of the Villain model
    also yields an EFIID for the Coulomb gas.
\end{enumerate}

\subsection{Proof of Theorems~\ref{thm:analyticity_XY}, \ref{thm:Coulomb_analytic}, and \ref{thm:analyticity_SWM} conditional on Theorem~\ref{thm:EFIID_SWM_XY}}
\label{subsec:conditional_proofs_2}

We now explain how the existence of an EFIID (Theorem~\ref{thm:EFIID_SWM_XY}) implies
analyticity of the free energy for the XY model and the SWM (Theorems~\ref{thm:analyticity_XY} and~\ref{thm:analyticity_SWM}).
Similarly, the EFIID for the Villain model (Theorem~\ref{thm:VILLAIN_EFIID})
leads to analyticity of the free energy of the Villain model (Theorem~\ref{thm:VILLAIN_ANALYTICITY}),
and thus to analyticity of the free energy of the Coulomb gas (Theorem~\ref{thm:Coulomb_analytic}).

We record the following general result.

\begin{theorem}[EFIID implies analyticity \cite{Ott_Analyticity_Ising,Ott_Analyticity_RCM}]
    \label{thm:EFIID_implies_analyticity}
 Let $\mu$ be a stationary random field on $S^{\Z^d}$.
 Write $\sigma\sim\mu$ for the random element.
 Consider a finite set $A\subset\Z^d$
 and a bounded observable $F_A$ that is $\sigma|_A$-measurable.
 If $\mu$ is an EFIID, then there exists a constant $\delta=\delta(\mu,F_A)>0$ such that the function
 \[
    \tilde F_A:(-\delta,\delta)\to\R,\,\epsilon\mapsto
    \lim_{n\to\infty}\frac1{|\Lambda_n|}\log\mu\Big[ \exp\Big(\eps\sum\nolimits_{x\in\Lambda_n}F_A(\tau_x\sigma)\Big)\Big]
 \]
 is analytic. 
 Here $\tau_x(\sigma)$ means that we translate the configuration $\sigma$ by $x$.
 \end{theorem}

 The proof is entirely contained in~\cite{Ott_Analyticity_Ising,Ott_Analyticity_RCM},
 although the theorem is stated in a less general form.
 We now describe how Theorem~\ref{thm:EFIID_implies_analyticity} implies Theorems~\ref{thm:analyticity_XY}, \ref{thm:Coulomb_analytic}, and~\ref{thm:analyticity_SWM}.
 A proof sketch of Theorem~\ref{thm:EFIID_implies_analyticity} is given at the end of this section.

\begin{proof}[Proof of analyticity of the free energy]
    We show that the free energy may be seen as the analytic function appearing in Theorem~\ref{thm:EFIID_implies_analyticity}.
    We do this for the SWM model; the other cases are similar.
    Fix $d$ and $\beta$.
    Let $A:=\{0,e_1,\ldots,e_d\}$,
    and
    \[
        F_A(\alpha) := \sum_{i=1}^d (\alpha_0 - \alpha_{e_i})^2.
    \]
    Then asymptotically in $n\to\infty$, we get
    \[
        \sup_{\alpha}\Big|H_{\SWM,\Lambda_n,\beta}^0(\alpha) - \sum\nolimits_{x\in\Lambda_n}F_A(\tau_x\alpha)\Big| =
        o(|\Lambda_n|),
    \]
    and so, for $|\epsilon|<\delta$,
    \[
        \tilde F_A(\eps) = \lim_{n\to\infty}\frac1{|\Lambda_n|}\log\mu_{\SWM,\beta}^0[e^{-\epsilon H_{\SWM,\Lambda_n}^0}]
        = \mathbf{\mathfrak{f}}_{\SWM}(\beta+\epsilon)-\mathbf{\mathfrak{f}}_{\SWM}(\beta).
    \]
    The equality on the right is standard, for models where the potential on each edges remains
    uniformly bounded (which is indeed the case for all models considered in this article).
    This proves analyticity of $\mathbf{\mathfrak{f}}_{\SWM}$ on $(\beta-\delta,\beta+\delta)$.

    The same proof holds for the XY model, using the measure $\mu^{+1}_{\XY,\beta}$ when $\beta < \beta_{c}(d)$. 
\end{proof}

Finally, we briefly discuss the proof of Theorem~\ref{thm:EFIID_implies_analyticity}.

\begin{proof}[Sketch of proof of Theorem~\ref{thm:EFIID_implies_analyticity}]
The existence of an EFIID for the measure, together with the coarse-graining procedure in~\cite{Ott_Analyticity_RCM}, imply the existence of a \emph{dependency encoding measure} in the sense of~\cite[Theorem 4.1]{Ott_Analyticity_RCM}. We point out that actually the condition of being EFIID is more restrictive than the existence of a dependency encoding measure~\emph{à la}~\cite{Ott_Analyticity_RCM}.  
The remainder of the proof is routine, as a convergent cluster expansion of the quantity
\[ G_n(\eps) := \mu\big[\exp\big(\eps\sum_{x \in \Lambda n} F_A(\tau_x \sigma )\big)\big] \]
may be produced with the method of~\cite{Ott_Analyticity_RCM} when $\eps$ is sufficiently small. This uses the fact that $F_A$ is bounded. 
Once the convergent cluster expansion is produced, the rest of the proof is routine --- the cluster expansion produces an analytic extension of $\log G_n$ on a small disk which is uniformly bounded in $n$, and Vitali's theorem is used to get the convergence and analyticity of its limit when $n$ tends to infinity.
For using Vitali's convergence theorem, one needs to check the convergence of $\log G_n$ on a subinterval of the real line containing 0.
Here our setting is slightly more general than~\cite{Ott_Analyticity_RCM}, but the convergence is ensured by the fact that $\mu$ is an EFIID, and classical sub-aditivity arguments.
We refer to~\cite{Ott_Analyticity_Ising,Ott_Analyticity_RCM} for details. 
\end{proof}

%% file: sections_new/SWM/main.tex
\input{sections_new/SWM/01_intro.tex}
\input{sections_new/SWM/02_glauber.tex}

\input{sections_new/SWM/03_spacetime.tex}

\input{sections_new/SWM/04_conclusion_efiid.tex}
\input{sections_new/SWM/05_mixing.tex}

%% file: sections_new/SWM/01_intro.tex
\section{SWM: Proof overview}

In many different situations, factors of i.i.d.\ are constructed
out of Glauber dynamics using a coupling-from-the-past (CFTP) procedure~\cite{vdb_steif_cftp,haggstrom_steif_cftp,MR3486171,spinka_cftp}.
Monotonicity of the Glauber dynamic is often a crucial ingredient
to make the procedure work.
The local set $\calL_u$ in the EFIID can be constructed explicitly,
by analysing how far back in the past we must look in space-time
to determine the value of $\phi(X)_u$.

In our case, we run into a new issue:
since the SWM takes real values, the natural candidate
for the Glauber dynamic (using the canonical grand coupling of all probability distributions
on $\R$) will never exactly couple in finite time.
Thus, one must look back infinitely far in space-time,
and the local sets are therefore too big to yield an EFIID.

We circumvent this issue in three steps (explained later in further detail).
\begin{itemize}
    \item First, we propose a new Glauber dynamic for the SWM,
    which satisfies the following properties:
    \begin{enumerate}
        \item\label{prop:mono} The monotonicity property mentioned above,
        \item\label{prop:digits} When the vertex $u$ is resampled, with very large probability,
        the digits in the decimal expansion of $\alpha_u$ after the first
        $k$ digits, are sampled uniformly at random independently of everything else,
        \item\label{prop:markov} The Markov property.
    \end{enumerate}
    Formal definitions may be found in Subsection~\ref{subsec:abstract_glauber}.
    Property~\ref{prop:digits} is the key modification, as it allows us to couple the dynamics up to the first $k$ digits in finite time, and then use the randomness in the remaining digits to achieve an exact coupling.
    \item Second, we show that the original Gibbs measure satisfies some
    very strong spatial mixing properties.
    \item Third, we prove that the mixing together with Properties \ref{prop:mono} and~\ref{prop:digits} of the new dynamic, imply that the Glauber dynamic has good space-time mixing.
    This proof closely follows the work of Harel and Spinka~\cite{harel_spinka}
    but we need to make some slight modifications to deal with the
    real-valued nature of the model.
    \item Finally, we combine space-time mixing of the Glauber dynamic
    with Property~\ref{prop:markov} to construct an EFIID representation of the infinite-volume SWM.
\end{itemize}

The remaining sections are organised as follows.
Spatial mixing of the Gibbs measure is stated here, but proved at the end, in Section~\ref{sec:spatial_mixing_SWM}.
In Section~\ref{sec:glauber}, we introduce the new Glauber dynamic,
and recall some general statement concerning coupling from the past.
Section~\ref{sec:spacetime} adapts the Harel--Spinka argument to our setup.
Finally, in Section~\ref{sec:efiid_SWM}, we use the space-time mixing of the Glauber dynamic and the Markov property to construct the EFIID representation of the infinite-volume SWM, and conclude the proof of Theorem~\ref{thm:EFIID_SWM_XY} for the square-well model. 

The main ingredient for proving the above-mentioned fast space-time mixing of the Glauber dynamics is the following mixing result, whose proof is postponed to Section~\ref{sec:spatial_mixing_SWM}. 
In what follows, when $\mu$ and $\nu$ are two probability measures on a measurable space $(\Omega, \calF)$, we use the notation
\[ \Vert \nu - \mu \Vert_{\mathrm{TV}}\] 
to denote the usual \emph{total variation distance} between $\mu$ and $\nu$.

\begin{lemma}[Spatial mixing of the Gibbs measure of the SWM]
    \label{lem:spatial_mixing_SWM}
    Fix the dimension $d\in\Z_{\geq 1}$ and the inverse temperature $\beta\in\R_{\geq 0}$.
    Then there exist constants $C<\infty$ and $c>0$ such that
    \[
        \Big\|\mu_{\Lambda_{2n},\beta}^{+1}|_{\Lambda_n}-\mu_{\Lambda_{2n},\beta}^{-1}|_{\Lambda_n}\Big\|_{\operatorname{TV}}
        \leq Ce^{-cn}
    \]
    for all $n\in\Z_{\geq 1}$.
\end{lemma}

%% file: sections_new/SWM/02_glauber.tex
\section{SWM: Glauber dynamic}
\label{sec:glauber}

This section is organised as follows:
\begin{itemize}
    \item In Subsection~\ref{subsec:abstract_glauber} we define Glauber dynamics and their abstract properties,
    \item In Subsection~\ref{subsec:construction_glauber_SWM} we construct such a Glauber dynamic for the SWM,
    \item In Subsection~\ref{subsec:cftp} we recall some relevant information on \emph{coupling from the past}.
\end{itemize}
We shall always work in the following setting:
$G=(V,E)$ is a transitive locally finite graph,
and $S$ a measurable space.
In the case of the SWM, $G$ will be $\Z^d$ and $S$ will be the interval $[-1,1]$.

\subsection{Abstract description of the Glauber dynamic}
\label{subsec:abstract_glauber}

The idea of the Glauber dynamic is to construct a Markov chain converging
to the Gibbs measure, by performing local updates at each vertex on the configuration.
The local updates are performed using an independent source of randomness.
In the abstract setting, this sample space for this independent source of
randomness is denoted by $\Sigma$,
a generic element is denoted $\iota\in\Sigma$,
and the probability measure is denoted $\diff\iota$.
A technical requirement for the measures $\mu$ to be sampled via Glauber dynamics
is that $\mu$ should have the \emph{finite energy property}, that means that for any $v \in V$, $\sigma \in S^V$, 
the conditional density of
$\mu( \sigma_v = \cdot ~\vert~\sigma_{u\neq v})$  
is strictly positive almost everywhere on $S$.
This is not a problem for us, as this property is clearly true in all the models considered.

\begin{definition}[Glauber dynamic]
    \label{def:glauber}
    Let $\mu$ denote a stationary measure on $S^V$
    with finite energy.
    A \emph{Glauber dynamic} for $\mu$ is a function
    \[
        R:S^V\times V\times\Sigma \to S^V
    \]
    with the following properties.
    \begin{itemize}
        \item\textbf{Local update.}
        For any $\alpha\in S^V$, $u\in V$ and $\iota\in\Sigma$, we have $R(\alpha,u,\iota)_v=\alpha_v$ for any $v\neq u$.
        \item\textbf{Consistency.}
        For any $\alpha\in S^V$ and $u\in V$, the distribution of $R(\alpha,u,\iota)_u$ in $\diff\iota$ is $\mu[\blank|(\alpha_v)_{v\neq u}]$.
        \item\textbf{Equivariance.}
        For any $\alpha\in S^V$, $u\in V$, $\iota\in\Sigma$ and $\theta\in\Theta$, we have $R(\alpha,u,\iota)\circ\theta=R(\alpha\circ\theta,u\circ\theta,\iota)$.
    \end{itemize}
    We shall also write $R_{u,\iota}:S^V\to S^V,\,\alpha\mapsto R(\alpha,u,\iota)$ for any $u\in V$ and $\iota\in\Sigma$.
\end{definition}

In addition to the above properties,
we shall quickly list some other properties that may or may not hold true for a given Glauber dynamic.
Of course, all these properties shall eventually apply to the Glauber dynamic that we construct in the next subsection.

\begin{definition}[Property~\ref{prop:mono}: Monotonicity]
    \label{def:prop:mono}
    Suppose that $S$ is a partially ordered set.
    Then we call a Glauber dynamic $R$ \emph{monotone} if for any $u\in V$ and $\iota\in\Sigma$,
    the map $R_{u,\iota}:S^V\to S^V$ preserves the partial order at the spin $u$.
    This means that
    $R_{u,\iota}(\alpha)_u\preceq R_{u,\iota}(\alpha')_u$ for any $\alpha,\alpha'\in S^V$ such that $\alpha_v\preceq\alpha'_v$ for all $v\in V$.
\end{definition}

    For any $x\in\R$, write 
    $\lfloor x\rfloor_k:=10^{-k}\lfloor 10^k x\rfloor$
    and
    $\{x\}_k:=x-\lfloor x\rfloor_k$.

\begin{definition}[Property~\ref{prop:digits}: Matching digits]
    \label{def:prop:digits}
    Suppose that $S\subset\R$.
    Consider fixed constants $\epsilon\in[0,1]$ and $k\in\Z_{\geq 0}$.
    We say that a Glauber dynamic $R$ \emph{matches digits up to $(\epsilon,k)$}
    if we may find an event $M\subset\Sigma$ of $\diff\iota$-probability at least $1-\epsilon$ such that:
    \begin{itemize}
        \item For any $u\in V$ and $\iota\in M$, we have
        \[
            \big\{
                \{R_{u,\iota}(\alpha)_u\}_k=\{R_{u,\iota}(\alpha')_u\}_k \text{ for all $\alpha,\alpha'\in S^V$}
            \big\}
        \]
        \item For any $u\in V$ and $\alpha\in S^V$,
        the event $M$ and the random variable $\lfloor R_{u,\iota}(\alpha)_u\rfloor_k$
        are $\diff\iota$-independent.
    \end{itemize}
    The event $M$ is called the \emph{matching event}.
\end{definition}

\begin{definition}[Property~\ref{prop:markov}: Markov property]
    \label{def:prop:markov}
    We say that a Glauber dynamic $R$ has the \emph{Markov property} if for any $\alpha\in S^V$, $u\in V$ and $\iota\in\Sigma$,
    the value of $R(\alpha,u,\iota)_u$ is measurable with respect to $\alpha|_{\calN_u}$ and $\iota$,
    where $\calN_u$ denotes the set of neighbours of $u$ in $G$.
\end{definition}

\subsection{Construction of the Glauber dynamic}
\label{subsec:construction_glauber_SWM}

\begin{lemma}[Glauber dynamic for the SWM with matching digits]
    \label{lem:glauber_SWM}
    Fix any locally finite transitive graph $G=(V,E)$ and any $\beta\in\R_{\geq 0}$.
    Then for any $\epsilon>0$, there exists a constant $k\in\Z_{\geq 0}$ and 
    a Glauber dynamic $R$ for the SWM on $G$ at inverse temperature $\beta$,
    which has Properties~\ref{prop:mono}--\ref{prop:markov},
    matching digits up to $(\epsilon,k)$.
\end{lemma}

The proof of the lemma requires two simple intermediate results.
To state them, we first need a notion of \emph{stochastic domination}.

\begin{definition}
    \label{def:stochastic_domination}
    Let $\mu$ and $\nu$ denote two probability measures on $\R$.
    Then we say that $\mu$ is \emph{stochastically dominated} by $\nu$
    if their cumulative distribution functions $F_\mu$ and $F_\nu$ satisfy $F_\mu\geq F_\nu$.
\end{definition}

\begin{lemma}[Grand coupling]
    \label{lem:grand_coupling}
    Consider the set $\calM$ of probability measures on $\R$.
    Then we may construct a coupling of all distributions in $\calM$ such that if $\mu$ is stochastically dominated by $\nu$, then the $\mu$-distributed random variable is almost surely less than or equal to the $\nu$-distributed random variable.
    More precisely, we require that we have a family of random variables $(X_\mu)_{\mu\in\calM}$ such that $X_\mu\sim\mu$ for each $\mu\in\calM$
    and such that $X_\mu\leq X_\nu$ whenever $\mu$ is stochastically dominated by $\nu$.
\end{lemma}

\begin{proof}
    Sample a uniform random variable $U$ on $[0,1]$,
    and set $X_\mu:=F_\mu^{-1}(U)$ for each $\mu\in\calM$.
    If $F_\mu^{-1}(U)$ is multi-valued then we break ties by taking the infimum of this set.
\end{proof}

\begin{lemma}[Grand coupling for almost uniform measures]
    \label{lem:grand_coupling_almost_uniform}
    Let $\epsilon>0$ and consider the set $\calM$
    of probability measures $\mu$ on $[0,1]$ such that
    $\mu\geq (1-\epsilon)\cdot\diffi x$ (the Lebesgue measure on $[0,1]$)
    as measures.
    Then all distributions in $\calM$ can be coupled together in such a way that:
    \begin{itemize}
        \item If $\mu$ is stochastically dominated by $\nu$, then $X_\mu\leq X_\nu$,
        \item With probability $1-\epsilon$, all the random variables $(X_\mu)_{\mu\in\calM}$ are equal.
    \end{itemize}
\end{lemma}

\begin{proof}
    Consider the cumulative distribution functions $F_\mu:[0,1]\to[0,1]$ of the measures $\mu\in\calM$.
    By assumption, the functions
    \[
        \tilde F_\mu:[0,1]\to[0,1],\,x\mapsto\frac{F_\mu(x)-(1-\epsilon)x}{\epsilon}
    \]
    are also cumulative distribution functions.
    Moreover, the operation $F_\mu\mapsto\tilde F_\mu$ preserves the partial order of stochastic domination
    (this is easy to see from the explicit expression).
    The desired coupling is now constructed as follows.
    First, sample a uniform random variable $U$ on $[0,1]$
    and another Bernoulli random variable $B$ with parameter $\epsilon$, independently of each other.
    \begin{itemize}
        \item If $B=0$, then we set $X_\mu:=U$ for each $\mu\in\calM$.
        \item If $B=1$, then we set $X_\mu:=\inf\tilde F_\mu^{-1}(U)$ for each $\mu\in\calM$.
    \end{itemize}
    This coupling satisfies the desired properties.
\end{proof}

\begin{proof}[Proof of Lemma~\ref{lem:glauber_SWM}]
    Let $D$ denote the degree of $G$.
    Fix $u\in V$.
    For any configuration $\alpha\in[-1,1]^{V\setminus \{u\}}$,
    let $\nu_\alpha$ denote the law of $\alpha_u$ in $\mu_{\{u\},\beta}^{\alpha}$.
    By definition, it is immediate that $\nu_\alpha$ is given by the normal distribution 
    \begin{equation}
        \label{eq:nu_alpha}
        \calN\Big(\frac1D\sum_{v\sim u}\alpha_v,\frac1{2\beta D}\Big),
    \end{equation}
    conditioned to be in the interval $[-1,1]$.
    The map $\alpha\mapsto\nu_\alpha$ has the following properties:
    \begin{itemize}
        \item The law $\nu_\alpha$ only depends on $\alpha|_{\calN_u}$,
        \item The law $\nu_\alpha$ is stochastically increasing in $\alpha$.
    \end{itemize}
    For the construction of the Glauber dynamic,
    it now suffices to construct a grand coupling of all measures $(\nu_\alpha)_{\alpha}$
    preserving the monotonicity and matching the digits.

    We do so in two steps.
    \begin{itemize}
        \item First, we consider the grand coupling of $(\nu_\alpha)_{\alpha}$ given by Lemma~\ref{lem:grand_coupling}:
        we first sample a uniform random variable $U'$ on $[0,1]$,
        and then set $X_\alpha':=\inf F_\alpha^{-1}(U')$.
        \item Next, we would like to \emph{resample} the digits after the first $k$ ones.
        More precisely, we want to couple the distributions
        \[
            \nu_\alpha':=\nu_\alpha[\blank|\{\lfloor 10^k X_\alpha\rfloor=\lfloor 10^k X_\alpha'\rfloor\}].
        \]
        Again, the map $\alpha\mapsto\nu_\alpha'$ is stochastically increasing in $\alpha$.
        Moreover, $\nu_\alpha'$ is just given by the distribution in Equation~\eqref{eq:nu_alpha} conditioned to be in the interval
        \[\Big[10^{-k}\lfloor 10^k X_\alpha'\rfloor,10^{-k}(\lfloor 10^k X_\alpha'\rfloor+1)\Big).\]
        If $k$ is large enough, then each measure $\nu_\alpha'$ is bigger (in the sense of measures)
        than $(1-\epsilon)$ times the uniform distribution on the conditioning interval.
        We may then apply Lemma~\ref{lem:grand_coupling_almost_uniform} to construct a grand coupling of the measures $\nu_\alpha'$
        preserving the monotonicity and matching the digits up to $(\epsilon,k)$.
    \end{itemize}
    Since this Glauber dynamic never inspects the values of $\alpha$ outside of $\calN_u$, it also has the Markov property.
\end{proof}

\subsection{Coupling from the past}
\label{subsec:cftp}
Above, we discussed how to update a single spin $\alpha_u$.
In practice, we want to update all spins many times in order to mix to the Gibbs measure.
We want to encode the information about this continuous Markov chain in terms of a set $\Pi\subset V\times\Sigma\times\R_{\leq 0}$,
where the first coordinate is the vertex to be updated, the second coordinate is the randomness used for the update,
and the third coordinate is the time of the update.
In this context, we call $V\times\R_{\leq 0}$ the \emph{space-time}.
We say that $\Pi$ has \emph{collisions} if two distinct points in $\Pi$ have the same time coordinate,
and we say that $\Pi$ is \emph{locally finite} if for any $u\in V$ and $t\in\R_{\leq 0}$,
only finitely many points of $\Pi$ have the form $(u,\iota,s)$ with $s\in [t,0]$.
Write
\[
    \Omega:=\big\{\Pi\subset V\times\Sigma\times\R_{\leq 0} : \text{$\Pi$ is locally finite and has no collisions}\big\}.
\]
In practice, $\Pi$ is sampled according to a Poisson point process on $V\times\Sigma\times\R_{\leq 0}$,
in which case $\Pi$ almost surely belongs to $\Omega$.

For any $\Pi\in\Omega$ and any bounded space-time subset $B\subset V\times\R_{\leq 0}$, we write
\[
    \Pi(B):=\{(u,\iota,t)\in\Pi : (u,t)\in B\}.
\]
Moreover, we think of $\Pi(B)$ as an ordered set, by ordering the points in $\Pi(B)$ according to their time coordinate,
with the lowest (most negative) time coming first (this is well-defined since $\Pi$ has no collisions).
Finally, we write
\[
    R_{\Pi(B)}:= (R_{u_n,\iota_n}\circ\cdots\circ R_{u_1,\iota_1}):S^V\to S^V ,\,\alpha\mapsto R_{u_n,\iota_n}\circ\cdots\circ R_{u_1,\iota_1}(\alpha),
\]
where $((u_i,\iota_i,t_i))_{i}$ is the ordered enumerate of the triples in $\Pi(B)$.
We now record a number of useful properties of this construction, that goes back to~\cite{propp_wilson}.

\begin{theorem}[Coupling from the past]
    \label{thm:cftp}
    Consider a Glauber dynamic $R$ for a measure $\mu$ on $S^V$.
    Suppose that $S$ is partially ordered and that $R$ is monotone.
    Suppose moreover $S$ has a largest element $\alpha^+$ and a smallest element $\alpha^-$.
    We also write $\alpha^\pm$ for the corresponding elements in $S^V$.
    Then for any fixed $\Pi\in\Omega$:
    \begin{itemize}
        \item For any bounded $B\subset V\times\R_{\leq 0}$, the map $R_{\Pi(B)}:S^V\to S^V$ is monotone,
        \item The element $R_{\Pi(\Lambda\times[-t,0])}(\alpha^+)$ is decreasing in $\Lambda$ and $t$,
        \item The element $R_{\Pi(\Lambda\times[-t,0])}(\alpha^-)$ is increasing in $\Lambda$ and $t$.
    \end{itemize}
    In particular, if
    \[R_{\Pi(\Lambda\times[-t,0])}(\alpha^+)_u=R_{\Pi(\Lambda\times[-t,0])}(\alpha^-)_u\]
    then this value will not change if we replace $\alpha^+$ by any other $\alpha\in S^V$,
    $\Lambda$ by any larger set, and $t$ by any larger time.
\end{theorem}

If $\Pi$ is a random set, then we shall also write:
\begin{itemize}
    \item $R^{\Lambda,-t}:=R_{\Pi(\Lambda\times[-t,0])}$,
    \item $R^{\Lambda,-t,-s}:=R_{\Pi(\Lambda\times[-t,-s])}$.
\end{itemize}

In what follows, we shall use this construction (Glauber dynamics coupled through the coupling-from-the-past).
In particular, $\alpha^+$ (resp. $\alpha^-$) will refer to the configuration constant equal to $1$ (resp. $-1$) on $\Z^d$.

%% file: sections_new/SWM/03_spacetime.tex
\section{SWM: Space-time mixing}
\label{sec:spacetime}

The goal of this section is to prove that the mixing property stated in Lemma~\ref{lem:spatial_mixing_SWM} implies that the Glauber dynamics defined above mixes both in space and in time. 
Our main goal is to prove the following space-time mixing estimate, stating that with high probability, the maximal and minimal dynamics for $\alpha$ in the coupling previously constructed remain equal for a positive proportion of the available space-time when the boundary conditions are sufficiently far away in space-time. 
\begin{proposition}\label{prop:space_time_mixing}
There exists $\delta > 0$ such that:
\[
\P[ \bigcap_{-r \in [-\delta n, 0]}\{   R^{\Lambda_{2n}, -n, -r }(\alpha^+)|_{\Lambda_n} = R^{\Lambda_{2n}, -n, -r }(\alpha^-)|_{\Lambda_n}  \}] \underset{n\rightarrow \infty}{\longrightarrow} 1.
\]
\end{proposition}

We divide the proof of this statement in two subsections. 
The first one is dedicated to reducing the proof to a key statement, namely Proposition~\ref{prop:subpolynomial_decay_coupling}.
The second subsection adapts an argument due to Harel and Spinka~\cite{harel_spinka}, which itself is generalisation of an original argument of~\cite{martinelli_olivieri}, to provide a proof of this key statement.

\subsection{Reduction to Proposition~\ref{prop:subpolynomial_decay_coupling}}

We start by defining some ``coupling events'' that we shall appeal to reapeteadly in what follows. 

\begin{definition}
For any $n \geq 0$ and $-t < -s < 0$, define the following coupling events.
\begin{multline*} \mathsf{C}(n,t,s) := \{ R^{\Lambda_{n}, -t, -s}(\alpha^+)_0 = R^{\Lambda_n, -t, -s}(\alpha^-)_0  \} \\ \text{ and }~\mathsf{C}(n,t,\geq s) := \bigcap_{-r \in [-s, 0]} \mathsf{C}(n,t,r), \end{multline*}
and 
\begin{multline*} \lfloor \mathsf{C}(n,t,s) \rfloor_k := \{ \lfloor R^{\Lambda_n(v), -t, -s}(\alpha^+)_0\rfloor_k = \lfloor R^{\Lambda_n(v), -t, -s}(\alpha^-)_0 \rfloor_k  \} \\ \text{ and }~\lfloor\mathsf{C}(n,t,\geq s)\rfloor_k := \bigcap_{-r \in [-s, 0]}  \lfloor\mathsf{C}(n,t,r)\rfloor_k . \end{multline*}

Finally, when $s = 0$, we remove it from the notation \emph{e.g.} we write:
\[  \mathsf{C}(n,t,0) := \mathsf{C}(n,t).  \]

Observe that those events are the ones appearing in Theorem~\ref{thm:cftp}, which explains their name of ``coupling events''.
We will also name their complementary events by replacing $\mathsf{C}$ by $\mathsf{NC}$ (``non-coupling'') for the four events just defined.  
\end{definition}

We first state the key statement that will allow us to prove Proposition~\ref{prop:space_time_mixing}.

\begin{proposition}\label{prop:subpolynomial_decay_coupling}
There exists $\delta > 0$ such that for any $n \geq 0$, and $v \in \Z^d$ and any exponent $\beta > 0$, it is the case that 
\[
n^\beta \P[\mathsf{NC}(n, n(1-\delta))] \underset{n\rightarrow \infty}{\longrightarrow} 0. 
\]
\end{proposition}

The goal of the remainder of this subsection is to derive Proposition~\ref{prop:space_time_mixing} from Proposition~\ref{prop:subpolynomial_decay_coupling}, that is to prove that a small fixed time probability of non-coupling implies a small probability of non-coupling during a time interval.
As Proposition~\ref{prop:subpolynomial_decay_coupling} implies a subpolynomial decay of the fixed time probability of non-coupling, we can be quite rough and perform a straightforward union bound, as explained in the following lemma. 

\begin{lemma}\label{lem:coupling_long_time}
There exist $c > 0$ and $\delta > 0$ such that for any $n \geq 0$, 
\[
\P[ \mathsf{NC}(n,n,\geq \delta n) ] \leq 100 \delta n  \P[ \mathsf{NC}(n,n(1-\delta)) ] + \exp(- cn). 
\]
The same result holds when replacing $\mathsf{NC}(n,n,\geq \delta n)$ (resp. $ \P[ \mathsf{NC}(n,n, \delta n) ] $) by $\lfloor\mathsf{NC}(n,n,\delta n)\rfloor_k$ (resp. $\lfloor\mathsf{NC}(n,n,\geq \delta n)\rfloor_k$).
\end{lemma}

\begin{proof}
In the time interval $[-\delta n, 0]$, the spin at 0 is resampled $\mathsf P$ times, where $\mathsf P$ is a Poisson variable of expectation $\delta n$. 
Each time that the spin at 0 is resampled, the probability of the non-coupling event is bounded by $\P[ \mathsf{NC}(n,n,\delta n) ] = \P[\mathsf{NC}(n,(1-\delta)n)]$ due to the monotonicity of the dynamics in time. 
The result follows by a union bound together with a classical Chernoff-type bound for the tails of Poisson variables. 
The proof also applies in the case of the coupling of the $k$-th first digits of $\alpha$. 
\end{proof}

We now explain how the proof of Proposition~\ref{prop:space_time_mixing} follows from Proposition~\ref{prop:subpolynomial_decay_coupling}.
This is nothing but a quite crude union bound.

\begin{proof}[Proof of Proposition~\ref{prop:space_time_mixing}]
Let $\delta$ be given by Proposition~\ref{prop:subpolynomial_decay_coupling}. 
Write 
\begin{multline*}
\P[\big\{ \bigcap_{-r \in [-\delta n, 0]}\{   R^{\Lambda_{2n}, -n, -r }(\alpha^+)|_{\Lambda_n} = R^{\Lambda_{2n}, -2n, -r }(\alpha^-)|_{\Lambda_n}  \} \big\}^c] 
\\ \leq n^d \sup_{v \in \Lambda_{n}}\P[ \mathsf{NC}_v(2n, n, \geq \delta n )] 
\leq n^d \P[\mathsf{NC}(n,n,\geq \delta n)]]
 \\ \leq n^{d}(100\delta n \P[ \mathsf{NC}(n, n(1-\delta) )] + \exp(-c\delta n)).
\end{multline*}
We used a union bound for the first inequality, monotonicity in space for the second inequality, and Lemma~\ref{lem:coupling_long_time} at the third line.
 By Proposition~\ref{prop:subpolynomial_decay_coupling}, this quantity tends to 0, which concludes the proof. 
\end{proof}

\subsection{The Harel-Spinka type argument}
We turn to the proof of Proposition~\ref{prop:subpolynomial_decay_coupling}.
Similarly to~\cite{harel_spinka}, our strategy is to provide a renormalization inequality for the coupling probability of the maximal and minimal dynamics. 
However, the task is complicated by the infiniteness of the spin space, and the proof crucially relies on the ``$(\eps, k)$-matching of the digits'' property satisfied by the dynamic. 

We will need three preparatory lemmas. 
The first two ones are a standard fact about disagreement probabilities of ordered random variables, and a
 basic calculus estimate that we shall use to establish the superpolynomial decay of $\P[\mathsf{NC}(n, n(1-\delta))]$.
 
\begin{lemma}[Lemma 15,~\cite{harel_spinka}]\label{lem:TV}
Let $X$ and $Y$ be two random variables taking values in a totally ordered and finite spin space $S$, such that $X \leq Y$ almost surely. Then,
\[ \P[X \neq Y] \leq (|S| -1)\Vert X-Y \Vert_{\mathrm{TV}}. \] 
\end{lemma}

\begin{lemma}\label{lem:tool_superpolynomial_decay}
Let $(\Psi(n))_{n\geq 0}$ be a positive sequence that we assume monotone decreasing to zero.
Assume that there exists $\kappa > 0$, $C,c > 0$, and a polynomial $P$ such that for any $1 \leq m \leq n$, 
\begin{equation}\label{equ:lem_renormalisation}
\Psi(2n) \leq P(m)\Psi((1-\kappa)n)^2 + C\e^{-cm}. 
\end{equation}
Then, $\Psi(n)$ tends superpolynomially fast to 0, that is, for any $A>0$, 
\[n^A \Psi(n) \underset{n\rightarrow \infty}{\longrightarrow} 0. \]
\end{lemma}

\begin{proof}
We basically repeat the argument of~\cite[Lemma 17]{harel_spinka}.
Call $\alpha = \tfrac{2}{1-\kappa}$.
We first rewrite the hypothesis as 
\[
\Psi(\alpha n) \leq P(m)\Psi(n)^2 + C\e^{-cm},
\]
valid for any $(1-\kappa)m \leq n$. 
By naming $a_n := - \log\Psi(n)$, we observe that 
\begin{equation}\label{equ:proof_computation_superpolynomial} a_{\alpha n} \geq 2cm - \log C \text{ for any }  m \leq n \text{ such that } 
cm + \log P(m) \leq 2a_n.  \end{equation}

Now observe that either $a_n \geq c(1-\kappa)^{-1}n/2$ for infinitely many values of $n$, either any solution of $cm + \log P(m) $ satisfies $(1-\kappa)m \leq n$. 
Observe that the former case implies exponential decay for $\Psi$; we focus on the latter. 
For any $\eps > 0$, observe that $cm + \log P(m) \leq 2x$ for all $m \leq (2/c -\eps)x$, and large $x$.
As $a_n$ tends to $+\infty$ by assumption, evaluate the latter expression in $x=a_n$  and use~\eqref{equ:proof_computation_superpolynomial} to obtain $a_{\alpha n} \geq c(2/c - \eps)a_n - \log C \geq (2-c\eps - \eps)a_n$.
We conclude that 
\[\lim a_{\alpha^n}2^{-\delta n}=\infty \qquad \text{for any } \delta \in (0,1). \]
It is clear that in that case, $\Psi(n)$ tends to 0 faster than any polynomial. 
\end{proof}

The third preparatory lemma states that, up to reducing linearly the time-space window, in the Glauber dynamics the probability of coupling only the first $k$ digits is comparable to the probability of coupling all the digits.  
This lemma would be false if we were to consider the ``classical'' Glauber dynamics for $\alpha$: we crucially use the property that our dynamics has the ``matching of the digits'' property. 

\begin{lemma}\label{lem:tree_eps_perco}
For  $\eps > 0$ small enough, for any $\delta \in (0,1)$, there exist $c_\eps, C_\eps > 0$ and $k \geq 0$ such that 
\begin{itemize}
\item The dynamics has the $(\eps, k)$-matching of the digits property.
\item For any $n\geq 0$, 
\[
\P[ \mathsf{NC}(n,n) ] \leq \exp(-c_\eps n) + C_\eps\P[\lfloor \mathsf{NC}((1-\delta)n,(1-\delta)n) \rfloor_k].
\] 
\end{itemize}
\end{lemma}

%
\begin{proof}
For a space-time point $(-t,v) \in \R_{\leq 0} \times\Z^d$, we construct its ``matching tree'' $\calT^M(-t,v)$ with the following recursive procedure.
Start from the space-time point $(-t,v)$, and denote by $- t_1$ the smallest negative time below $-t$ at which the spin $\alpha_v$ was resampled. 
We examine the occurence of the matching event $M_{(- t_1, v)}$ in the definition of Property~\ref{prop:digits}, and note that it is independent of the $k$-th first digits of the updated value of the spin at $(-t_1, v)$ by construction. 
Now, 
\begin{itemize}
\item If $M_{(-t_1, v)}$ occurs, we add $(-t_1, v)$ to the tree. 
In that case, we shall say that $(t_1, v)$ is a \emph{leaf} of the tree. 
\item If $M_{(-t_1, v)}^c$ occurs, we add $(-t_1,v)$ and $\calT^M(-t_1, u)$ to the tree, for each $u\sim v$.
\end{itemize}
This provides a well-defined notion of matching tree.
We are interested in the matching tree of the point $(0,0)$ that be call $\calT^M$ in short. 
Also call $\partial \calT^M$ the set of leaves of the matching tree of $(0,0)$.
We claim the following two properties: 

 \begin{itemize}
 \item On the event $\{ \calT^M \cap \partial( (-n,0] \times \Lambda_n ) = \emptyset \}$, the random variable $R^{\Lambda_n, -n}(\alpha^\pm)_0$ is measurable with respect to 
 \begin{itemize}
 \item The data of $R^{\Lambda_n, -n, -t'}(\alpha^\pm)_{u'}$, for all $(-t', u')\in \partial \calT^M$. 
 \item The data of $\Pi(\Lambda_n \times (-n,0]).$
 \end{itemize}
 This is due to the Markov property of the Glauber dynamics: at each resampling time, the dynamics inspect the values of the neighbours of the spin to be resampled, and uses the Poisson randomness to resample. A value on the boundary of the boxe is never explored by assumption. 
 
 \item On the event that the two truncated dynamics coincide on the leaves set $\partial \calT^M $, \emph{i.e.} on the formal event
 \[ 
 \{ \calT^M \cap\partial( (-n,0] \times \Lambda_n ) = \emptyset \} \cap \{ \bigcap_{(-t', u) \in  \partial \calT^M} \lfloor\mathsf{C}_u(n, n, -t')\rfloor_k \},
  \]
  then it is the case that 
 \[ R^{\Lambda_n, -n}(\alpha^+)_0 = R^{\Lambda_n, -n}(\alpha^-)_0. \]
Indeed, at the leaves of the tree, $\alpha^+$ and $\alpha^-$ coincide: their $k$ first digits are the same by assumption, and by definition of the leaves of the tree, their further digits are matched. The claim is then a consequence of the first item, with the fact that the dynamics never explores values on the boundary of the box by assumption.
 \end{itemize}
 
Now recall that $0<\delta < 1$, and assume that $\eps > 0$ is sufficiently small, and $k = k(\eps)$ is chosen so that: 
\begin{itemize}
\item The dynamics is $(\eps, k)$-matching the digits.

\item There exists $c_\eps > 0$ such that 
\[ \P[ \calT^M \cap (\Lambda_{\delta n}\times (-\delta n, 0] )^c \neq \emptyset  ] \leq \exp(- c_\eps \delta n ).\]
\end{itemize}

Such a choice can always be made, by comparing the law of the offspring of the matching tree with that of a subcritical branching process in continuous time and using classical results about those processes, see~\cite{athreya2004branching} for instance.
It is now easy to conclude. 
Indeed, the event $\mathsf{NC}(n,n,0)$ is partitioned as wether the event $\calT^M  \cap ((-\delta n,0] \times \Lambda_{\delta n})^c \neq \emptyset $ occurs or not. 
We obtain

\begin{multline*}
\P[\mathsf{NC}(n,n)]  \leq \exp(-c_\eps \delta n) \\
+ \P[ \{\calT^M  \subset ((-\delta n,0] \times \Lambda_{\delta n}) \} ,  \bigcup_{(-t', u) \in \partial_{\mathrm{int}} \calT^M} \lfloor \mathsf{NC}_u(n,n, -t') \rfloor_k] .
\end{multline*}
Call $C_\eps$ the expected size of the leaves set of $\calT^M$ (which is finite by the fact that we chose $\eps >0$ so small that this tree is dominated by a subcritical branching process),  and observe that by monotonicity in space and in time, $\P[\lfloor \mathsf{NC}_u(n,n, -t') \rfloor_k] \leq \P[\lfloor \mathsf{NC}((1-\delta)n ,n, -\delta n) \rfloor_k] $ when $t' \leq \delta n$ and $u\in \Lambda_{\delta n} .$ 
Using the independence between the occurence of the matching event and the first $k$ digits of the dynamics, we obtain
\begin{align*}
\P[\mathsf{NC}(n,n)]  &\leq \exp(-c_\eps \delta n) + C_\eps \P[\lfloor \mathsf{NC}((1-\delta)n,n, -\delta n) \rfloor_k]\\
&=  \exp(-c_\eps \delta n) + C_\eps \P[\lfloor \mathsf{NC}((1-\delta)n,(1-\delta)n) \rfloor_k].
\end{align*}
which concludes the proof. 
 \end{proof}


We are ready to prove Proposition~\ref{prop:subpolynomial_decay_coupling}. 

\begin{proof}[Proof of Proposition~\ref{prop:subpolynomial_decay_coupling}]
Fix $\delta \in (0,1)$.
In light of Lemma~\ref{lem:tree_eps_perco}, one has

\begin{align*}
\P[\mathsf{NC}(n,n,\delta n)] &\leq \P[\mathsf{NC}((1-\delta)n, (1-\delta)n )] \\
&\leq \exp(-c_\eps \delta n) + \P[\lfloor \mathsf{NC}((1-\delta)^2n, (1-\delta)^2n) \rfloor_k]. 
\end{align*}

It is thus sufficient to argue that the quantity $\Psi(n) :=  \P[\lfloor \mathsf{NC}(n, n) \rfloor_k]$ decays superpolynomially fast.
We are now in the context of a finite spin space, and will mimic the Harel--Spinka argument.
However the absence of ``temporal Markov property'' for the truncated field requires yet another argument in the proof, that we shall make explicit. 

Our goal is to use Lemma~\ref{lem:tool_superpolynomial_decay} and thus to argue that $\Psi$ satisfies a renormalization inequality. 
Fix $\kappa > 0$, the value of which will be set later on. 
Also define 
\[\phi(n,s) := \P[ \lfloor \mathsf{NC}(n,s) \rfloor_k ].\]

Introduce two variables $\xi^+, \xi^-$, such that $\xi^+ \sim \mu^+_{\Lambda_m}, \xi^- \sim \mu^-_{\Lambda_m}$ and $\xi^- \leq \xi^+$ almost surely. 
We first write, using Lemma~\ref{lem:TV}: 
\begin{equation}\label{equ:phi_plus_phi_minus}
\phi(n,s) \leq \phi^+(n,s) + \phi^-(n,s) + 10^k  \Vert \mu^+_{\Lambda_n}(\lfloor \alpha_0 \rfloor_k \in \cdot) - \mu^-_{\Lambda_n}(\lfloor \alpha_0 \rfloor_k \in \cdot)   \Vert_{\mathrm{TV}},
\end{equation}
where 
\[
\phi^\pm(n,s) =  \P[ R^{\Lambda_{n}, s}(\alpha^\pm)_0 \neq R^{\Lambda_{n}, s}(\xi^\pm)_0]. 
\]
By the mixing property of Lemma~\ref{lem:spatial_mixing_SWM}, the total variation term is upper bounded by $\e^{-cn}$.

We will now prove that for any $m, n \in \N$ any $t, s \in \R_{>0}$ such that $t \geq n$, the following holds:
\begin{multline}\label{equ:target_phi}
\phi^{\pm}(m+n, t+s) \leq  10^k\exp(-cm) + m^d\exp(-c_\eps \kappa n) \\ + C_\eps m^d\phi(n,s)\phi((1-\kappa)n, (1-\kappa)n). 
\end{multline}
The exact values in this inequality are not so important. 
Indeed, to conclude, it will be sufficient to observe that for $s=t=n$ and $m \leq n$,~\eqref{equ:phi_plus_phi_minus} and~\eqref{equ:target_phi} give 
\[ \phi(m+n, 2n) \leq 2C_\eps m^d\Psi(n)\Psi((1-\kappa) n) + m^d\e^{-cn} + C\e^{-cm},\]
and as $m \leq n$, the term $m^d\e^{-cn}$ can be absorbed by $C\e^{-cm}$, altering the values of $c, C > 0$.
Monotonicity in space-time allows to bound 
\[\Psi(n) \leq \Psi((1-\kappa)n), \] so that $\Psi$ satisfies the remormalisation equation~\eqref{equ:lem_renormalisation}.
As is is clear that $\Psi$ is monotone decreasing to 0, the proof is complete.

We thus focus on proving~\eqref{equ:target_phi} for $\phi^+$, and fix $n,m,s,t$ and $\kappa$ accordingly. 
Introduce an independent copy of the Poisson point process on the volume $\Lambda_{n+m}$ in the time interval $[-t, 0]$, and denote by $\tilde R^{\Lambda_{n+m}, t}$ the associated Glauber map.
Standard independence properties of Poisson point processes imply that: 
 \[
 \phi^+(n+m, t+s) = \P[ \lfloor R^{\Lambda_{n+m}, s}(\tilde R^{\Lambda_{n+m}, t}(\alpha^+))_0 \rfloor_k \neq \lfloor R^{\Lambda_{n+m}, s}(\tilde R^{\Lambda_{n+m}, t}(\xi^+))_0 \rfloor_k ]. 
 \]
 We decompose over the event:
 \[ \calE = \bigcap_{v \in \Lambda_m} \{ \tilde R^{\Lambda_{n+m}, t}(\alpha^+)_v = \tilde R^{\Lambda_{n+m}, t}(\xi^+)_v  \}. \]
 
 Indeed, observe that:
 
 \begin{itemize}
 \item If $\calE$ occurs, then $R^{\Lambda_{n+m}, t}(\alpha^+)$ reached the invariant measure on $\Lambda_m$. 
 Consequently, conditioned on that event, the worst situation for coupling is if $R^{\Lambda_{n+m}, t}(\alpha^+)$ (resp. $R^{\Lambda_{n+m}, t}(\xi^+)$) is maximal (resp. minimal) outside of the box $\Lambda_m$.
 Call $\zeta^{+, +}$ (resp. $\zeta^{+, -}$) the configuration equal to $R^{\Lambda_{n+m}, t}(\xi^+)$ in $\Lambda_m$ and maximal (resp. minimal) outside of it. 
 This discussion yields (we still work conditionally on $\calE$, and write $\phi_\calE(\cdot, \cdot)$ for the conditional probability)
 \[
 \phi_{\calE}(n+m, t+s) \leq  \P[ \lfloor R^{\Lambda_{m}, s}(\zeta^{+, +})_0\rfloor_k \neq   \lfloor R^{\Lambda_{m}, s}(\zeta^{+, -})_0\rfloor_k ].
 \]
 But observe that $\zeta^{+, -} \geq \zeta^{-,-}$, where $\zeta^{-,-}$ is constructed by the same procedure with $R^{\Lambda_{n+m}, t}(\xi^-)$ instead of $R^{\Lambda_{n+m}, t}(\xi^+)$. 
 This is due to the monotonicity in the coupled dynamics. 
 Thus,
 \[
 \phi_\calE(n+m, r+s) \leq \P[ \lfloor R^{\Lambda_{m}, s}(\zeta^{+, +})_0\rfloor_k \neq   \lfloor R^{\Lambda_{m}, s}(\zeta^{-, -})_0\rfloor_k ]
 \]
 By the fact that $\zeta^{+,+}$ (resp. $\zeta^{-,-}$) is invariant for the dynamics with $+$ (resp. $-$) boundary conditions on $\Lambda_m$ and using Lemmas~\ref{lem:TV} and~\ref{lem:spatial_mixing_SWM}, we obtain
 \begin{equation}
 \label{equ:renormalization_1}
  \phi_\calE(n+m, r+s) \leq 10^k \Vert \mu^{+}_{\Lambda_m}(\alpha_0 \in \cdot) - \mu^{-}_{\Lambda_m}(\alpha_0 \in \cdot)  \Vert_{\mathrm{TV}} \leq 10^k \exp(-cm). 
 \end{equation}
 
 \item Now, if $\calE$ does not occur, in which case we write $\phi_{\calE^c}(\cdot, \cdot)$ for the conditional probability, then the worst situation for coupling is if $R^{\Lambda_{n+m}, t}(\alpha^+)$ (resp. $R^{\Lambda_{n+m}, t}(\xi^+)$) is maximal (resp. minimal) in $\Lambda_{n+m}$. 
 In that case, we may bound the conditional probability of failing to couple at time 0 by $\phi(n+m, s)$. 
 Thus
 \begin{equation}\label{equ:renormalization_2}
 \phi_{\calE^c}(m+n, t+s) \leq \phi(n+m, s).
 \end{equation}
 \end{itemize}
 
 To conclude it remains to prove an upper bound on $\P[\calE^c]$. 
 Here a new argument is required.
 Fix $\kappa > 0$ to be the size of a ``buffer zone'' that we are going to use for applying Lemma~\ref{lem:tree_eps_perco}. 
 Indeed, we observe that, due to the choice $t \geq n$ and by monotonicity in time and space,
 \[
 \P[\calE^c] \leq m^d \sup_{v \in \Lambda_m}  \P[\mathsf{NC}_v(n+m, n)] \leq m^d \P[\mathsf{NC}(n, n)].
 \]
 Now, Lemma~\ref{lem:tree_eps_perco} implies that
 \begin{align}  \label{equ:renormalization_3}
 \P[\mathsf{NC}_v(n, n)] &\leq \exp(-c_\eps \kappa n) + C_\eps\phi((1-\kappa)n, (1-\kappa)n)  \\ 
 \end{align}
 
 Putting~\eqref{equ:renormalization_1},~\eqref{equ:renormalization_2}, and~\eqref{equ:renormalization_3} together yields:
 \begin{align*}
 \phi(n+m, t+s)& \leq 10^k \exp(-cm) + m^d\phi(n+m, s)( \exp(-c_\eps \kappa n) + C_\eps\phi((1-\kappa)n, (1-\kappa)n)) \\
 &\leq 10^k \exp(-cm) +  m^d\exp(-c_\eps \kappa n) + C_\eps m^d \phi(n,s)\phi((1-\kappa)n, (1-\kappa)n).
 \end{align*}
 This is exactly~\eqref{equ:target_phi}, which concludes the proof.
\end{proof}

%% file: sections_new/SWM/04_conclusion_efiid.tex
\section{SWM: Construction of the EFIID}
\label{sec:efiid_SWM}

\subsection{Strong supercriticality of the process of mixed boxes}

Let $\delta > 0$ be the quantity given by Proposition~\ref{prop:space_time_mixing}.
Introduce $L \in \N$. 
This number will be thought about as the coarse-graining scale of the space-time set $\R^- \times \Z^d$.
By convenience, assume that $L\delta^{-1} := n_L$ is always an integer.

We first introduce a notion of \emph{mixed point}. 

\begin{definition}\label{def:mixed_box}
For any $(t,x) \in L\cdot(\Z_{\leq 0} \times \Z^d)$, we say that $(t,x)$ is $(L,\delta)$-mixed if the event 
\[ \bigcap_{-r \in [-L(t+1), -L t] } \{R^{\Lambda_{2n_L}(v), -(t+1)n_L, -r}(\alpha^+ )|_{\Lambda_{n_L}(v)} = R^{\Lambda_{2n_L}(v), -(t+1)n_L, -r}(\alpha^- )|_{\Lambda_{n_L}(v)} \}\]
occurs. 
\end{definition}

In other words, we ask that the maximal and minimal dynamics remain coupled in a box of size $n_L$ during a time $L$, when the boundary conditions are at space time distance $2n_L$ from $(t,v)$. 
The scales are exactly chosen so as to apply Proposition~\ref{prop:space_time_mixing}.

We start by observing that the environment of mixed boxes can be made extremely subcritical when $L$ is large. 
Introduce the site percolation process on $L \cdot(\Z_{\leq 0} \times \Z^d)$ defined by 
\[\Theta^{\delta, L}_{(t,x)} = \ind{ \{ (t,x) \text{ is }(L, \delta)-\text{ mixed}\}}.\] 

\begin{lemma}\label{lem:stoch_dom}
There exists a sequence $\eps_L$ tending to 0 when $L\rightarrow \infty$ such that the law of $\Theta^{\delta, L}$ stochastically dominates $\P^L_{1-\eps_L}$, where $\P^L_{1-\eps_L}$ is the law of a Bernoulli bond percolation of parameter $1-\eps_L$ on the lattice $L\cdot (\Z_{\leq 0} \times \Z^d)$. 
\end{lemma}

\begin{proof}
For any $(t,x) \in L \cdot(\Z_{\leq 0} \times \Z^d)$, observe that the event $\{\Theta^{L,\delta}_{(t,x)} = 1\}$ is measurable with respect to
\[ \Pi(\Lambda_{2n_L}(v), [-(t+1)n_L, tL] ). \]
This implies that the process $\Theta$ is a $2\delta^{-1} + 1$ dependent process on $L\cdot (\Z_{\leq 0} \times \Z^d)$, endowed with nearest-neighbour connectivity. 

Furthermore, by the choice of $n_L$, Proposition~\ref{prop:space_time_mixing} implies that for any $(t,v) \in L\cdot (\Z_{\leq 0} \times \Z^d)$,
\[\P[\Theta^{L, \delta}_{(t,x)} = 1] \underset{L\rightarrow \infty}{\longrightarrow} 1.\]
A classical result of~\cite{liggett_schonmann} then implies the result. 
\end{proof}

We also gather the properties of the set of mixed boxes that we shall use to construct the EFIID representation of the model. 

\begin{lemma}\label{lem:property_good_boxes}
Let $(t,x) \in L\cdot (\Z_{\leq 0} \times \Z^d)$ such that $(t,x)$ is $(L,\delta)$-mixed. 
Then, the following statements are true:
\begin{itemize}
\item For any $r \in [-L(t+1), -L t]$, the distributions of $R^{\Lambda_{2n_L}(v), -(t+1)n_L, -r}(\alpha^\pm)|_{\Lambda_{n_L}(x)}$ are given by the measure $\mu^{\SWM}_{\Lambda_{2n_L}}(\cdot|_{\Lambda_{n_L}(x)})$.
\item For any $r \in [-L(t+1), -L t]$, the random variables $R^{\Lambda_{2n_L}(x), -(t+1)n_L, -r}(\alpha^\pm)|_{\Lambda_{n_L}(x)}$ are independent of the Poisson process outside of the set $\Lambda_{2n_L}(x) \times [-(t+1)n_L, tL] $. 
\end{itemize}
\end{lemma}

\begin{proof}
Both of these statements come from the properties of the coupling from the past stated in Theorem~\ref{thm:cftp}.
\end{proof}

\subsection{Construction of the local sets for $\Theta$}

For any $L \geq 0$, we endow the space-time set $L\cdot(\Z_{\leq 0} \times \Z^d)$ with the usual $*$-connectivity: two vertices are considered neighbours when their $\ell^\infty$ norm is equal to $L$ in the underlying graph $\Z_{\leq 0} \times \Z^d$.
In a $\{0,1\}$-valued site percolation process on $L\cdot(\Z_{\leq 0} \times \Z^d)$, we denote by $\calC$ the 1-cluster of the vertex 0 for the $*$-connectivity, and by $\calC^*$ the 0-cluster of the vertex 0 for the $*$-connectivity.
We start by fixing the value of $L$ that we shall use to construct the local sets. 

\begin{lemma}\label{lem:peierls}
There exists a value of $L \geq 0$ large enough and two constants $c,C>0$ such that for any $n \geq 1$, 
\[
\P^L_{1-\eps_L}[|\calC^*| > n] \leq C\exp(-cn). 
\]
\end{lemma}

\begin{proof}
This follows from a standard Peierls arguments for Bernoulli bond percolation of small parameter $\eps_L$ in the graph $L\cdot(\Z_{\leq 0} \times \Z^d)$.
\end{proof}

{\bf The value of $L$ is now fixed and given by Lemma~\ref{lem:peierls}, as well as the values of the constants $c,C$ > 0.}
We will also drop the dependency in $L$ and in $\delta$ in the notation $\Theta^{\delta, L}$. 

For an integer $N \geq 0$, we also introduce the following usual notion of \emph{$N$-external complement}. 
For a connected set $\mathsf{C} \subset L\cdot(\Z_{\leq 0} \times \Z^d)$ containing 0, we define its $N$-external complement by:
\[ \partial^N_{\mathrm{ext}}\mathsf{C} := \{ y \in L\cdot(\Z_{\leq 0} \times \Z^d)\setminus \mathsf{C}, \inf_{x\in\mathsf{C}}\Vert x-y\Vert_\infty > L N  \text{  and  } y \overset{L\cdot(\Z_{\leq 0} \times \Z^d)\setminus \mathsf{C}}{\longleftrightarrow} \infty. \}\]
where we recall that the infinite norm is taken with respect to the underlying graph $\Z_{\leq 0}\times \Z^d$. 
The notation $ y \overset{L\cdot(\Z_{\leq 0} \times \Z^d)\setminus \mathsf{C}}{\longleftrightarrow} \infty$ means that there exists an infinite self-avoiding nearest-neighbour path in $L\cdot(\Z_{\leq 0} \times \Z^d)\setminus \mathsf{C}$ emanating from $y$. 

We now call $\calC^*$ the 0-cluster of the vertex (0,0) in the process $\Theta$.
The key input for constructing the local sets for $\alpha_0$ is the following observation. 

\begin{lemma}\label{lem:decoupling}
The random variable $R^{\infty, -\infty, 0}(\alpha^+)_0$ is independent of the Poisson point process $\Pi$ restricted to the set $\partial^{2n_L}_{\mathrm{ext}} \mathcal{C}^* \times \R_{< 0}$, where we recall that $n_L = \lceil \delta^{-1}L \rceil$. 
\end{lemma}

\begin{proof}
The lemma follows by the observation that the external boundary of $\calC^*$ consists in a space-time surface of mixed boxes shielding 0 from $\partial^{n_L}_{\mathrm{ext}}C^*$.
When updating the state of a point on the interior of that surface, the Glauber rule inspects the value of the field:
\begin{itemize}
\item Either on a vertex contained on the inside of the surface, in which case it does not depend on the Poisson randomness outside of it,
\item  Or on a vertex contained in the surface, in which case by Lemma~\ref{lem:property_good_boxes}, it is also independent of the Poisson process on $\partial^{2n_L}_{\mathrm{ext}} \mathcal{C}^* \times \R_{< 0}$. 
\end{itemize}  
As $(0,0)$ is contained in the space-time surface, the proof is complete.  
\end{proof}

We conclude by explaining how this property implies Theorem~\ref{thm:EFIID_SWM_XY} in the case of the SWM. 

\begin{proof}[Proof of Theorem~\ref{thm:EFIID_SWM_XY} for the SWM]
The EFIID is constructed as follows. 
The field $(X_v)_{v\in\Z^d}$ is given by the data of the Poisson point process $\Pi(\{v\}, \R_{\leq 0})_{v\in\Z^d}$, 
and the map $\phi$ is given by 
$\phi(X)_v := R^{\infty, \infty, 0}(\alpha^+)_v$, 
for any $v\in\Z^d$. 
Moreover, for any $v \in\Z^d$, define the set $\calL_v := \Z^d \setminus \pi (\partial^{n_L}_{\mathrm{ext}}C^*_v)$, 
where $\pi : \Z_{\leq 0}\times\Z^d \to \Z^d, (t,x) \mapsto x$ is the spatial projection.

\begin{itemize}

 \item By construction, $\phi$ is clearly measurable, equivariant with respect to the automorphism group of $\Z^d$ , and $\phi(X)$ is distributed as $\mu^{\mathsf{SWM}}$. 

\item That $\{\Theta^{L,\delta}_{(t,x)} = 1\}$ is measurable with respect to
\[ \Pi(\Lambda_{2n_L}(v), [-(t+1)n_L, tL] ). \]
implies that $\calL_v$ is a local set for $\Pi$. 

\item Lemma~\ref{lem:decoupling} implies that the value of $\phi(X)_v$ is measurable with respect to $X|_{\mathcal{L}_v}$.

\item Finally, by Lemma~\ref{lem:peierls} together with the stochastic domination given by Lemma~\ref{lem:stoch_dom}, 
\[ \P[|\mathcal{L}_v | > n| \leq C\exp(-cn).\]
\end{itemize}

Thus the measure $\mu^{\SWM}$ is a EFIDD.

\end{proof}

%% file: sections_new/SWM/05_mixing.tex
\section{SWM: Spatial mixing of the Gibbs measure}
\label{sec:spatial_mixing_SWM}

This section contains a proof of Lemma~\ref{lem:spatial_mixing_SWM}.
The quadratic interaction in the Hamiltonian of $\mu_{\Lambda_{2n},\beta}^{+1}$ makes
that we can the associated \emph{Brownian interpolation}.
Let $C_\Lambda\subset\R^d$ denote the \emph{cable graph} of $\Lambda\subset\Z^d$,
and write $\tilde\alpha$ for the Brownian interpolation of $\alpha$ on $C_\Lambda$.
The corresponding Gibbs measure is denoted $\tilde\mu_{\Lambda_{2n},\beta}^{+1}$.
Refer to~\cite{Lupu_2016_LoopClustersRandom} for details.
By flip-symmetry of Brownian motion, we get
\begin{align}
    \Big\|\mu_{\Lambda_{2n},\beta}^{+1}|_{\Lambda_n}-\mu_{\Lambda_{2n},\beta}^{-1}|_{\Lambda_n}\Big\|_{\operatorname{TV}}
    &\leq
    \tilde\mu_{\Lambda_{2n},\beta}^{+1}[\{
        \Lambda_n\xleftrightarrow{\{\tilde\alpha\neq 0\}} \partial\Lambda_{2n}
    \}]
    \\&\leq
    \sum_{x\in\partial\Lambda_n}
    \tilde\mu_{\Lambda_{2n},\beta}^{+1}[\{
        x\xleftrightarrow{\{\tilde\alpha\neq 0\}} \partial\Lambda_{2n}
    \}]
    \\&\asymp
        \sum_{x\in\partial\Lambda_n}
    \mu_{\Lambda_{2n},\beta}^{+1}[\alpha_x].
\end{align}
For Lemma~\ref{lem:spatial_mixing_SWM}, it suffices to prove that this last expression decays exponentially in $n$.
This is proved by comparison with a massive Gaussian free field.
More precisely, we define the massive Gaussian free field
$\nu_{\Lambda,\beta,m}^\zeta$ on $\Lambda\subset\Z^d$ with mass $m\in(0,\infty)$ and inverse temperature $\beta\in[0,\infty)$
as the measure on $\R^\Lambda$ with density
\[
\frac1{Z_{\Lambda,\beta,m}^\zeta} e^{-\beta H_\Lambda^\zeta(\alpha)-m\sum_{u\in\Lambda}\alpha_u^2}\diff\alpha.
\]
Since all interactions are quadratic, this measure is Gaussian,
which means that it is quite easy to make explicit calculations.
The desired result follows from the following two lemmas.

\begin{lemma}
    \label{lem:spatial_mixing_SWM_massive}
    Fix a dimension $d\geq 1$, an inverse temperature $\beta\in[0,\infty)$ and a mass $m\in(0,\infty)$.
    Then there exist a constant $c\in(0,\infty)$ such that for any finite $\Lambda\subset\Z^d$
    and for any $u\in\Lambda$, we have
    \[
        \nu_{\Lambda,\beta,m}^{+1}[\alpha_u]\leq e^{-c\operatorname{Distance}(u,\partial\Lambda)}.
    \]
\end{lemma}

\begin{lemma}
    \label{lem:spatial_mixing_SWM_comparison}
    Fix a dimension $d\geq 1$ and an inverse temperature $\beta\in[0,\infty)$.
    Then there exists a mass $m\in(0,\infty)$ such that for any finite $\Lambda\subset\Z^d$ and for any $u\in\Lambda$, we have
    \[
        \mu_{\Lambda,\beta}^{+1}[\alpha_u]\leq \nu_{\Lambda,\beta,m}^{+1}[\alpha_u].
    \]
\end{lemma}

Jointly the two lemmas clearly imply the desired exponential decay.
We provide a proof for both lemmas;
a version of the second lemma already appeared in the work of McBryan--Spencer~\cite{McBryanSpencer_1977_DecayCorrelationsInSOnsymmetric}.
We provide an alternative proof using the FKG inequality that may be of independent interest.
This proof makes it slightly easier to handle boundary conditions.

\begin{proof}[Proof of Lemma~\ref{lem:spatial_mixing_SWM_massive}]
    Let $X$ denote a simple random walk in $\Z^d$ started from $u$,
    which measure is written $\P$. 
    Let $T$ denote the first time that $X$ hits $\partial\Lambda$.
    Then the Gaussian structure implies that
    \[
        \nu_{\Lambda,\beta,m}^{+1}[\alpha_u]
        = \E[(1+\tfrac{m}{2d\beta})^{-T}]
        \leq (1+\tfrac{m}{2d\beta})^{-\operatorname{Distance}(u,\partial\Lambda)}.
    \]
    This is the desired result.
\end{proof}

\begin{proof}[Proof of Lemma~\ref{lem:spatial_mixing_SWM_comparison}]
    We present here a variation of the McBryan--Spencer argument.
    We present the argument in the most general context possible.
    Let $G$ denote a finite connected graph,
    and let $g\in V(G)$ denote a fixed vertex
    (we think of $g$ as the boundary of $G$,
    it is typically a vertex with a very large degree).

    Now consider the following setup.
    For any $u\in V(G)$, $\lambda_u$ denotes an arbitrary probability measure on $\R$
    that is symmetric under a sign flip.
    We write $\lambda:=\otimes_u\lambda_u$.
    Then, let $\phi_\lambda$ denote the probability measure on $\R^{V(G)}$
    with density
    \[
        \frac1{Z(\lambda)}e^{-\beta \sum_{uv\in E(G)} (\alpha(u)-\alpha(v))^2 }
        \diffi\lambda(\alpha).
    \]
    Now set $\lambda:=\lambda^A$ where $A\subset V(G)$,
    and
    \[
        \lambda^A_u:=
        \begin{cases}
            (\delta_{+1}+\delta_{-1})/2 &\text{if $u=g$,}\\
            U([-1,1]) &\text{if $u\neq g$ and $u\not\in A$,}\\
            \calN(0,1/2m)&\text{if $u\neq g$ and $u\in A$.}
        \end{cases}
    \]
    Now let $G$ denote the finite graph with vertex set  $V(G)=\Lambda\cup\{g\}$,
    and whose edge set is the one inherited from $\Lambda$.
        To finish the proof, it suffices to prove that
    \[
         \mu_{\Lambda,\beta}^{+1}[\alpha_v]
         =
         \phi_{\lambda^\emptyset}[\alpha_g\alpha_v]
         \leq
         \phi_{\lambda^{V(G)}}[\alpha_g\alpha_v]
         =
         \nu_{\Lambda,\beta,m}^{+1}[\alpha_v],
    \]
    where the equalities are obvious.
    To finish the proof, it suffices the demonstrate that
    \begin{equation}
        \label{eq:desired_domination_per_vertex}
        \phi_{\lambda^A}[\alpha_g\alpha_v]
        \leq
        \phi_{\lambda^{A\cup\{u\}}}[\alpha_g\alpha_v]
    \end{equation}
    for any $A\subset V(G)$ and for any $u\in \Lambda$
    for a suitably chosen $m=m(d,\beta)$.
    
    The rest of the proof is dedicated to deriving Equation~\eqref{eq:desired_domination_per_vertex}.
    By standard results on absolute-value-FKG,
    it is easy to see that Equation~\eqref{eq:desired_domination_per_vertex}
    holds true if
    the law of $|\alpha_u|$ under $\phi_{\lambda^A}$ is stochastically
    dominated by its law under $\phi_{\lambda^A\cup\{u\}}$.
    For this purpose, we compare its law in \emph{three} measures:
    the measures
    \[
        \phi:=\phi_{\lambda^A},\qquad
        \phi':=\phi_{\lambda^{A\cup\{u\}}}\big[\blank\big|\{|\alpha_u|\leq 1\}\big],
        \qquad\text{and}\qquad
        \phi'':=\phi_{\lambda^{A\cup\{u\}}}.
    \]
    It is easy to see how they are related:
    $\phi'$ is a conditioned version of $\phi''$,
    and $\phi'$ may be written as the measure $\phi$ with a Radon--Nikodym derivative
    proportional to
    \[
        e^{-m\alpha_u^2}.
    \]
    Suppose now that the conditioning event (in the definition of $\phi'$) has a $\phi''$-probability
    of at most $1-\epsilon$ for some $\epsilon=\epsilon(d,\beta)>0$.
    If $m=m(\epsilon)$ is sufficiently small, then the Radon--Nikodym derivative
    $\diff\phi'/\diff\phi$ satisfies
    \[
        \diff\phi'/\diff\phi \leq (1-\epsilon)^{-1}.
    \]
    It is then easy to see that the law of $|\alpha_u|$
    in $\phi$ is stochastically dominated by its law under $\phi''$.
    It suffices to find an appropriate $\epsilon$.

    Notice that by FKG, the \emph{conditional} probability of $\{|\alpha_u|\leq 1\}$
    is the \emph{largest} when $\alpha_t=0$ for all neighbours $t\sim u$,
    In that case, the law of $\alpha_u$ is $\calN(0,1/2(2d\beta+m))$,
    where $2d$ is the degree of $u$ in the graph $\Z^d$.
    In particular, the means that (when $m\leq 2d\beta$)
    \[
        \phi''[\{|\alpha_u|\leq 1\}]
        \leq
        \P[\{|X|\leq 1\}]
    \]
    where $X\sim\calN(0,1/8d\beta)$ in $\P$.
    This is the desired uniform bound in terms of $d$ and $\beta$.
\end{proof}

%% file: sections_new/xy/main.tex
\input{sections_new/xy/01_intro.tex}

\input{sections_new/xy/02_glauber.tex}

\input{sections_new/xy/03_spacetime.tex}
\input{sections_new/xy/04_conclusion_efiid.tex}

\input{sections_new/xy/05_mixing.tex}

%% file: sections_new/xy/01_intro.tex
\section{XY: Proof overview and monotone representation}

The spins of the XY model take values in the circle $\S^1\subset\C$.
This is not naturally a (partially) ordered set with a minimal and maximal element.
In this part, we first introduce a representation of the XY model that is appropriately partially ordered.
Then, we introduce a Glauber dynamic that works well with this partial ordering.
Once the Glauber dynamic has been established, we follow the proof steps of Part~\ref{part:SWM}.
The main difference lies in the fact that our Glauber dynamic does not have the Markov property
(Definition~\ref{def:prop:markov}), and we must perform some steps to replace the Markov property
by something slightly weaker (to be more precise:
we show that in equilibrium, the Glauber dynamic is Markov with respect to connected local sets
that have exponential tails).

\subsection{Monotone representation of the XY model}

The representation for the XY model is most easily introduced on finite graphs.
Let $G=(V,E)$ denote a finite graph and fix $\beta\in[0,\infty)$.
The \emph{XY model} on $G$ with inverse temperature $\beta$ is the probability measure $\mu_{G,\beta}$ on $\sigma\in(\S^1)^{V}$
with a density
\[
    \diff\mu_{G,\beta}(\sigma)
    \propto
    e^{-\sum_{uv\in E}\frac\beta2\|\sigma_u-\sigma_v\|_2^2}
    \cdot
    \diff\sigma
\]
with respect to the product Lebesgue measure.

Here, we consider the \emph{coordinate representation} of the XY model,
which was, to the best knowledge of the authors, first introduced by Chayes~\cite{Chayes_stiffness} and later studied in~\cite{DubedatFalconet_2022_RandomClustersVillain}.
Each spin $\sigma_u\in\S^1$ may be written in terms of its two coordinates
(or in terms of its real and imaginary part) via
\[
    \sigma_u = \xi_u\cos\alpha_u + i\cdot \zeta_u\sin\alpha_u\quad\text{where}
    \quad
    \xi_u,\zeta_u\in\{-1,+1\},\,\alpha_u\in[0,\pi/2].
\]
The decomposition is unique unless $\alpha_u\in\{0,\pi/2\}$
which clearly happens almost never in $\mu_{G,\beta}$ as the law of $\alpha_u$
has a density with respect to the Lebesgue measure.

By writing out the Hamiltonian in terms of the triple $(\alpha,\xi,\zeta)$,
it is easy to see that, conditional on $\alpha$, the spins $\xi$ and $\zeta$
behave like two independent Ising models, whose couplings depend on $\alpha$.
We may consider simultaneously the independent FK--Ising couplings
of the two independent Ising models with the associated Fortuin--Kasteleyn (FK) percolation models.
This leads to the coupling of $\alpha$ with two percolations, denoted $\omega$ and $\eta$ below.

Another way to obtain the same coupling $(\alpha,\omega,\eta)$ is by attaching
attaching a Brownian interpolation (in $\C$) to each edge $uv $ of the graph from $\sigma_u$
to $\sigma_v$,
and recording the edges of $G$ where this Brownian interpolation does \emph{not}
hit the imaginary axis (this is $\omega\subset E$) and the edges where the interpolation
does \emph{not} hit the real axis (this is $\eta\subset E$).
Both approaches (which are in fact equivalent) lead to the following lemma.

\begin{lemma}[Coordinate representation]\label{lemma_density_quintuple}
    Let $G$ be a finite graph and $\beta\in[0,\infty)$.
    Define $\mu_{G,\beta}'$ as the probability measure on
    \[
        \tau:=(\alpha,\omega,\eta)\in [0,\tfrac\pi2]^V\times\{0,1\}^E\times\{0,1\}^E:=\Pi_G
    \]
    with density
    \begin{align}
        &
        \diff\mu_{G,\beta}'(\tau)
        \propto
        \textstyle
        \left(\prod_{u\in V}\diff\alpha_u\right)
        \\
        &
        \qquad\qquad
        \textstyle
        \cdot
        2^{k(\omega)}
        \left(
            \prod_{uv\in \omega}
                (1-\e^{-2\beta(\cos\alpha_u)(\cos\alpha_v)})
            \middle)\middle(
            \prod_{uv\not\in\omega}
                e^{-2\beta(\cos\alpha_u)(\cos\alpha_v)}
            \right)
            \\&
            \qquad\qquad
            \textstyle
            \cdot
            2^{k(\eta)}
            \left(
                \prod_{uv\in \eta}
                     (1-\e^{-2\beta(\sin\alpha_u)(\sin\alpha_v)})
                \middle)\middle(
                \prod_{uv\not\in\eta}
                    e^{-2\beta(\sin\alpha_u)(\sin\alpha_v)}
                \right).
    \end{align}
    Here $k(\rho)$ denotes the number of connected components of $(V,\rho)$.

    Now obtain a sample $\sigma\in(\S^1)^V$ as follows:
    \begin{itemize}
        \item Sample $\tau\sim\mu_{G,\beta}'$,
        \item Sample $\xi\in\{-1,+1\}^V$ by flipping a fair coin for each connected component of $\omega$,
        \item Sample $\zeta\in\{-1,+1\}^V$ by flipping a fair coin for each connected component of $\eta$,
        \item Set $\sigma_u:=\xi_u\cos\alpha_u + i\cdot \zeta_u\sin\alpha_u$ for each $u\in V$.
    \end{itemize}
    Then $\sigma\sim\mu_{G,\beta}$.
\end{lemma}

\begin{proof}
    This straightforward calculation is left to the reader.
\end{proof}

To construct an EFIID for the XY model, it now suffices to:
\begin{itemize}
    \item Construct an EFIID for the measure the coordinate representation,
    \item Show that all the $\omega$-clusters and $\eta$-clusters have exponential tails.
\end{itemize}

To mimic the strategy of Part~\ref{part:SWM},
it is essential that we can construct a Glauber dynamic that has some monotonicity properties.
This is discussed now.

\begin{definition}[Ordering of triples]
    \label{def:ordering_triples}
    Consider two triples $\tau=(\alpha,\omega,\eta)$ and $\tau'=(\alpha',\omega',\eta')$.
    We say that the first is smaller than the second,
    and write $\tau\preceq\tau'$,
    whenever all of the following conditions hold:
    \[
    \alpha \leq \alpha', \qquad -\omega \leq -\omega', \qquad \eta \leq \eta',
    \]
    where the inequalities are understood componentwise.
    This defines a partial order on $\Pi_G$.
\end{definition}

\begin{lemma}[Monotonicity of angles]
    \label{lemma_monotonicity_angles}
    Fix a finite graph $G$, some $\beta\in[0,\infty)$.
    Let $u$ denote some vertex, and let $E_u$ denote the set of edges incident to $u$.
    Pick
    \[
        \tau'=(\alpha',\omega',\eta')\in [0,\pi/2]^{V\setminus\{u\}}\times\{0,1\}^{E\setminus E_u}\times\{0,1\}^{E\setminus E_u},
    \]
    Let $\mu_{\tau'}$ denote the measure
    $\mu_{G,\beta}$ conditioned on \[\{\tau|_{(V\setminus\{u\})\times (E\setminus E_u) \times (E\setminus E_u)}=\tau'\}.\]
    Then the density of $\alpha_u$ on $[0,\pi/2]$ under $\mu_{\tau'}$
      is proportional to
    \begin{align}
        \textstyle
        \left(
        \prod_{\Gamma\in C(\omega')}
        \textstyle
        2\cosh(\beta\cos\alpha_u\sum_{v \in \Gamma}\cos\alpha_v')
        \right)
        \left(
        \prod_{\Gamma\in C(\eta')}
        \textstyle
        2\cosh(\beta\sin\alpha_u\sum_{v\in \Gamma}\sin\alpha_v')
        \right),
\end{align}
where $C(\rho)$ is the partition of $\{v\in V:v\sim u\}$ into $\rho$-connected components.

In particular, if $\mu^{\alpha_u}_{\tau'}$ denotes the law of $\alpha_u$ under $\mu_{\tau'}$, then
\[
    \tau'\preceq\tau''
    \qquad
    \implies
    \qquad
    \mu^{\alpha_u}_{\tau'}\leqstoch\mu^{\alpha_u}_{\tau''}.
\]
\end{lemma}

\begin{proof}
    The formula for the density follows from Lemma~\ref{lemma_density_quintuple} via a calculation.
    For the stochastic domination, one simply checks the Holley criterion.
\end{proof}

\begin{lemma}[Monotonicity of the edges]
    \label{lemma_monotonicity_edges}
    Fix a finite graph $G$ and some $\beta\in[0,\infty)$.
    Let $E'\subset E$ and $uv\in E\setminus E'$.
    Define $G':=(V,E')$.
    For any $\tau'=(\alpha',\omega',\eta')\in\Pi_{G'}$, write
    $\mu^{\omega_{uv}}_{\tau'}$ and $\mu^{\eta_{uv}}_{\tau'}$ for the laws of $\omega_{uv}$ and $\eta_{uv}$ in the conditional measure
    \(
        \mu_{G,\beta}[\blank|\{\tau|_{G'}=\tau'\}]
    \).
    Then
    \[
        \tau'\preceq\tau''
        \qquad
        \implies
        \qquad
        (
        \mu_{\tau''}^{\omega_{uv}}\leqstoch\mu_{\tau'}^{\omega_{uv}}
        \quad\text{and}\quad
        \mu_{\tau'}^{\eta_{uv}}\leqstoch\mu_{\tau''}^{\eta_{uv}}
        )
    \]
Notice also that $\omega$ and $\eta$ are independent in this conditional measure.
Moreover, $\mu^{\omega_{uv}}_{\tau'}$ only depends on the values of $\alpha_u$, $\alpha_v$,
and the partition of $\{x\in V:xy\in E\setminus E'\}$ into $\omega'$-connected components.
The same holds true for $\mu^{\eta_{uv}}_{\tau'}$ (mutatis mutandis).
\end{lemma}

\begin{proof}
    Conditional on $\alpha$, the percolations $\omega$ and $\eta$ are just independent FK percolation models
    with cluster weight $q=2$ and with edge weights depending on $\alpha$.
    Moreover, the edge weights for $\omega$ are increasing in $\alpha$,
    while those for $\eta$ are decreasing in $\alpha$.
    The stochastic domination then simply follows from well-known results for FK percolation.
\end{proof}

\subsection{Statement of the spatial mixing estimate}

Finally, we state our external spatial mixing input that is proved at the end, in Section~\ref{sec:XY_spatial_mixing}.
The boundary conditions $+1,+i\in\S^1\subset\C$ are chosen to represent the minimal and maximal boundary
conditions respectively in the monotone representation introduced above.

\begin{lemma}[Spatial mixing of the Gibbs measure of the XY model]
    \label{lem:spatial_mixing_XY}
    Fix the dimension $d\in\Z_{\geq 1}$ and the inverse temperature $\beta\in[0,\beta_c(d))$.
    Then there exists constants $C<\infty$ and $c>0$ such that
    \begin{enumerate}
        \item For any $n\in\Z_{\geq 1}$,  
        \[\Big\|\mu_{\Lambda_{2n},\beta}^{+1}|_{\Lambda_n}-\mu_{\Lambda_{2n},\beta}^{+i}|_{\Lambda_n}\Big\|_{\operatorname{TV}}
        \leq \frac{C}5e^{-cn}.\]
        
          \item For any $n\in\Z_{\geq 1}$,
           \[     \sup_{\zeta\in\{+1,+i\},\,\nu\in\{\omega,\eta\}}
        \mu_{\Lambda_{2n},\beta}^{\zeta}[\{\Lambda_n\xleftrightarrow{\nu}\partial\Lambda_{2n}\}]
        \leq \frac{C}5e^{-cn}. \]
   
\end{enumerate}

\end{lemma}

%% file: sections_new/xy/02_glauber.tex
\section{XY: Glauber dynamic}
\label{sec:glauber_XY}

\subsection{Abstract description of the Glauber dynamic}

The following definition of a Glauber dynamic is essentially equivalent to Definition~\ref{def:glauber}.
We must make a small technical modification:
in Definition~\ref{def:glauber} we only encoded a spin $\alpha_u$ at each vertex,
while now we also encode some percolations around each vertex.
In the current section, we impose that a Glauber update resamples the value of $\alpha_u$
at each vertex $u$, as well as all percolation edges incident to $u$.

\begin{definition}[Glauber dynamic]
    \label{def:glauber_XY}
    Consider the square lattice graph $\Z^d$ (also denoted $G=(V,E)$).
    Let $\mu$ denote a stationary measure on $\Pi_G$.
    For any $u\in V$, let $G_u$ denote the subgraph of $G$ induced by the vertices $V\setminus\{u\}$.
    A \emph{Glauber dynamic} for $\mu$ is a function
    \[
        R:\Pi_G\times V\times\Sigma \to \Pi_G,
    \]
    with the following properties.
    \begin{itemize}
        \item\textbf{Local update.}
        For any $\tau\in \Pi_G$, $u\in V$ and $\iota\in\Sigma$, we have $R(\tau,u,\iota)|_{G_u}=\tau|_{G_u}$.
        \item\textbf{Consistency.}
        For any $\tau\in \Pi_G$ and $u\in V$, the distribution of $R(\tau,u,\iota)$ in $\diff\iota$ is $\mu[\blank|\tau|_{G_u}]$.
        \item\textbf{Equivariance.}
        For any $\tau\in \Pi_G$, $u\in V$, $\iota\in\Sigma$ and $\theta\in\Theta$, we have $R(\tau,u,\iota)\circ\theta=R(\tau\circ\theta,u\circ\theta,\iota)$.
    \end{itemize}
    We shall also write $R_{u,\iota}:\Pi_G\to\Pi_G,\,\tau\mapsto R(\tau,u,\iota)$ for any $u\in V$ and $\iota\in\Sigma$.
\end{definition}

We stress again that this definition is essentially equivalent to Definition~\ref{def:glauber}.
Next, we extend the properties of the Glauber dynamic from Definition~\ref{def:glauber} to the current setting.

\begin{definition}[Property~\ref{prop:mono}: Monotonicity]
    \label{def:prop:mono_XY}
    A Glauber dynamic $R$ is \emph{monotone} if for any $u\in V$ and $\iota\in\Sigma$,
    the map $R_{u,\iota}:\Pi_G\to\Pi_G$ preserves the partial order $\preceq$ of Definition~\ref{def:ordering_triples}.
\end{definition}

\begin{definition}[Property~\ref{prop:digits}: Matching digits]
    \label{def:prop:digits_XY}
    Consider fixed constants $\epsilon\in[0,1]$ and $k\in\Z_{\geq 0}$.
    We say that a Glauber dynamic $R$ \emph{matches digits up to $(\epsilon,k)$}
    if we may find an event $M\subset\Sigma$ of $\diff\iota$-probability at least $1-\epsilon$ such that:
    \begin{itemize}
        \item For any $u\in V$ and $\iota\in M$, we have
        \[
            \big\{
                \{R_{u,\iota}(\tau)_{1,u}\}_k=\{R_{u,\iota}(\tau')_{1,u}\}_k \text{ for all $\tau,\tau'\in \Pi_G$}
            \big\}
        \]
        \item For any $u\in V$ and $\tau\in\Pi_G$,
        the event $M$ and the random variable $\lfloor R_{u,\iota}(\tau)_{1,u}\rfloor_k$
        are $\diff\iota$-independent.
    \end{itemize}
    The event $M$ is called the \emph{matching event}.
\end{definition}

For defining the modification of the Markov property enjoyed by the dynamics, it is convenient to introduce the following notation.
For $\tau\in \Pi_G$ and $v \in V(G)$, denote by $\calC^\om_v(\tau)$ (resp. $\calC^\eta_v(\tau)$) the $\om$ (resp $\eta$)-cluster of $v$. 
Also call $\calC^\om_{\calN_v}(\tau)$ (resp. $\calC^\om_{\calN_v}(\tau)$) the union of the $\om$ (resp. $\eta$)-clusters of the neighbours of $v$.

\begin{definition}[Property~\ref{prop:markov}b: almost Markov property]
    
    \label{def:prop:markov_XY}
    We say that a Glauber dynamic $R$ has the \emph{almost Markov property} if for any $\tau\in\Pi_G$, $u\in V$ and $\iota\in\Sigma$,
    the value of $R(\tau,u,\iota)$ at $u$ and on the edges incident to $u$,
    is measurable with respect to $\iota$, together with:
    \begin{itemize}
         \item The data of $\calC^\om_{\mathcal{N}_u}(\tau|_{G_u})$,
        \item The data of $\calC^\eta_{\mathcal{N}_u}(\tau|_{G_u})$
        \item The data of $\alpha|_{\mathcal{N}(u)}$ in $\tau$.
    \end{itemize}
\end{definition}

Compared to Part~\ref{part:SWM} and the case of the SWM, the relaxation of a ``true'' Markov property~\ref{prop:markov} into the ``almost'' Markov property of Definition~\ref{def:prop:markov_XY} is an important change.
Indeed, the arguments relying on the Markov property need to be adapted. 
As an example for the reader, the ``shielding'' argument used to construct the local sets in Section~\ref{sec:efiid_SWM} is not valid anymore: even though a vertex is shielded by a surface of mixed boxes, 
the Glauber dynamics could look for information further away, through the existence of a very long connection in $\omega$ or $\eta$. 
However, those two percolations are proved to have exponentially small clusters, and the arguments can be adapted.

\subsection{Construction of the Glauber dynamic}

\begin{lemma}[Glauber dynamic for the XY model with matching digits]
    \label{lem:glauber_XY}
    Fix $d\in\Z_{\geq 1}$ and $\beta\in\R_{\geq 0}$.
    Then for any $\epsilon>0$, there exists a constant $k\in\Z_{\geq 0}$ and 
    a Glauber dynamic $R$ for the monotone representation of the XY model on $\Z^d$ at inverse temperature $\beta$,
    which has Properties~\ref{prop:mono}, \ref{prop:digits}, and~\ref{prop:markov}b,
    matching digits up to $(\epsilon,k)$.
\end{lemma}

\begin{proof}
    The Glauber dynamic is constructed as follows.
    First, one erases the value of $\alpha_u$, as well as $\omega$ and $\eta$ on all edges incident to $u$.
    Then, one first resamples $\alpha_u$ according to the conditional distribution of $\alpha_u$
    given by Lemma~\ref{lemma_monotonicity_angles}.
    This can be done in a monotone fashion (Lemma~\ref{lemma_monotonicity_angles})
    and with matching digits (by following the same strategy as in Subsection~\ref{subsec:construction_glauber_SWM}).
    Then, one simply needs to resample $\omega$ and $\eta$ on all edges incident to $u$ in a monotone fashion;
    this can be done thanks to Lemma~\ref{lemma_monotonicity_edges}.
    By the above construction, the Glauber dynamic clearly satisfies Properties~\ref{prop:mono} and~\ref{prop:digits}.
    Since the conditional law of the new values only depends on $\alpha$
    at the neighbours of $u$, as well as the connectivity of the neighbours of $u$
    in $\omega$ and $\eta$, the almost Markov property (Property~\ref{prop:markov}b) is also automatically satisfied.
\end{proof}

We can import the simultaneous coupling of the Glauber dynamics in any volume, time and boundary conditions given by the coupling-from-the-past of Theorem~\ref{thm:cftp}.
In what follows, the notation $R^{\Lambda, -t,-s}(\tau^{\pm})$ will be used to refer to this construction. 
It will also be convenient to distinguish between the law of the percolation processes $\om$ and $\eta$ and of the field $\alpha$ under this dynamics. 
This is done as follows.

\begin{definition} Fix a subgraph $G$ of $\Z^d$, a configuration $\tau \in \Pi_G$, and $-\infty \leq -t < -s \leq 0$. 
\begin{itemize}
\item The $\om$ (resp. $\eta$)-marginal of $R^{\Lambda, -t, -s}(\tau)$ will be denoted by $R_\om^{\Lambda, -t, -s}(\tau)$ (resp. $R_\eta^{\Lambda, -t, -s}(\tau)$). 
\item The $(\om,\eta)$-marginal of $R^{\Lambda, -t, -s}(\tau)$ will be denoted by $R_{(\om,\tau)}^{\Lambda, -t, -s}(\tau)$.
\item The $\alpha$-marginal of $R^{\Lambda, -t, -s}(\tau)$ will be denoted by $R_\alpha^{\Lambda, -t, -s}(\tau)$. 
\end{itemize}
\end{definition}

Our goal is now to import the structure used in the proof of Theorem~\ref{thm:EFIID_SWM_XY} in Part~\ref{part:SWM}.
The majority of the arguments is preserved, and we shall reprove and adapt the ones that use the exact Markov property.

%% file: sections_new/xy/03_spacetime.tex
\section{XY: Space-time mixing}
\label{sec:spacetime_XY}

Our goal is to prove the analog of Proposition~\ref{prop:space_time_mixing} with our new Glauber map. 
In what follows, with distinguish an edge $e_0$ adjacent to the vertex 0.
It will be convenient to write $\bfz = (0,e_0)$.

We also import all the notation regarding the coupling events $\mathsf{C}$ and $\mathsf{NC}$ introduced in Section~\ref{sec:spacetime}, with the correct adaptation. 
In particular, for any point $v \in \Lambda_n$, any edge $e\in \Lambda_n$,
\begin{itemize}
\item $\mathsf{C}(n,n)$ (resp. $\mathsf{NC}(n,n)$) will refer to $\mathsf{C}(n,n)_\bfz$ (resp. $\mathsf{NC}(n,n)_\bfz$).
\item $\mathsf{C}_{\om,\eta, e}(n,n)$ will refer to the coupling event at the edge $e$, only for the dynamics on $(\om,\eta)$ given by $R_{(\om, \eta)}^{\Lambda_n, -n, 0}(\tau^{\pm})$.
\item $\mathsf{C}_{\alpha, v}(n,n)$ (resp. $\lfloor\mathsf{C}_{\alpha, v}(n,n)\rfloor_k$)  will refer to the coupling event at $v$, only for the dynamics on $\alpha$ (resp. $\lfloor \alpha \rfloor_k$)  given by $R_{\alpha}^{\Lambda_n, -n, 0}(\tau^{\pm})$.
\item Finally, when considering the ``global'' truncated coupling event $\lfloor \mathsf{C}(n,n) \rfloor_k$ and its complement $\lfloor \mathsf{NC} (n,n)\rfloor_k$, the truncation only concerns the field $\alpha$, as $\om$ and $\eta$ already have a finite spin space. 
\end{itemize}

For convenience, we rewrite the target proposition. 

\begin{proposition}\label{prop:space_time_mixing_XY}
There exists $\delta > 0$ such that
\[
\P\big[ \bigcap_{-r\in [-\delta n, 0] } \{ R^{2n, -n, -r}(\tau^+)_{\bfz} \neq \{ R^{2n, -n, -r}(\tau^+)_\bfz  \} \big] \underset{n \rightarrow \infty}{\longrightarrow} 1.
\]
\end{proposition}

The proof of this statement requires an adequate adaptation. 
We first observe that when forgetting about the field $\alpha$, the argument of Harel--Spinka directly imply the mixing of the dynamics for the marginal on $(\om,\eta)$, as this random variable enters into their framework.\footnote{Actually,the authors of~\cite{harel_spinka} work in the much more general context of \emph{upwards-backwards specifications}, and in particular do not make any use of a form of Markov property.}
As, such the following proposition is a direct consequence of~\cite{harel_spinka}. 

\begin{proposition}\label{prop:space_time_mixing_omega_eta}
There exist $c, C > 0$ and $\delta>0$ such that for any $n\geq 0$, 
\[
\P[ \mathsf{NC}_{(\om,\eta)}(n, -n, \geq -\delta n)] \leq C\e^{-cn}.
\]
\end{proposition}

We stress on the fact that this statement does not \emph{a priori} say anything on $\alpha$. 

\begin{proof}[Proof of Proposition~\ref{prop:space_time_mixing_omega_eta}]
The statement for $\delta = 0$ directly from~\cite{harel_spinka} as the law of $(\om,\eta)$ enters their framework of upwards-downwards specifications, is monotone, and inherits the weak mixing property of the whole triple $(\alpha, \om,\eta)$. 
For extending it to the space-time box $\Lambda_n \times (-\delta n, 0]$, we use Lemma~\ref{lem:coupling_long_time} that does not rely on the Markov property of the model. 
\end{proof}

We now explain how to adapt the argument of Section~\ref{sec:spacetime}.
The careful reader might check that the only result of that section relying on the Markov property of the dynamics is Lemma~\ref{lem:tree_eps_perco}.
Indeed, the spin space of the truncated dynamics still is finite, and the Harel--Spinka argument does not rely on the Markov property. Thus, the proof of Proposition~\ref{prop:space_time_mixing_XY} boils down to the following lemma, that we prove right after.

\begin{lemma}~\label{lem:tree_eps_perco_XY}
For $\eps > 0$, for any $\delta\in (0,1)$, there exist $c'_\eps, C'_\eps > 0$ such that 
\begin{itemize}
\item The dynamics has the $(\eps, k)$ matching of the digits property 
\item For any $n\geq 0$, 
\[ \P[\mathsf{NC}(n,n)] \leq \exp(-c'_\eps n) + C'_\eps\P[ \lfloor \mathsf{NC}((1-\delta) n, (1-\delta)n)\rfloor_k ]. \]
\end{itemize} 
\end{lemma}

To prove this statement, we adapt the construction of the matching tree as to also incorporate all the points belonging to the $\om$ and $\eta$ clusters of the neighbours of a point to be resampled at each step. 
This will restore the actual Markov property, and we will check that this tree still has an exponential probability to die out in linear time and space. 
More formally, for a space-time point $(-t,v)\in \R_{\leq 0} \times \Z^d$, construct its ``Markov-matching tree'', still noted $\calT^M(-t,v)$, with the following recursive procedure. 

Start with the space-time point $(-t,v)$ and add it to the tree.
Next, denote by $-t_1$ the first negative time below $-t$ at which $v$ was resampled. 
Again we examine the occurence of the matching event $M(-t_1, v)$. 
\begin{itemize}
\item If $M(-t_1, v)$ occurs, we add $(-t_1, v)$ to the tree. As previously we say that this space-time point is a \emph{leaf} of the tree. 
\item If $M(-t_1, v)$ does not occur, then the following are added to the tree:
\begin{itemize}
\item The vertex $(-t_1, v)$.
\item The vertices of $\calC^\om_{\mathcal{N}_u}(\tau|_{G_u})$,
\item The vertices of $\calC^\eta_{\mathcal{N}_u}(\tau|_{G_u})$,
\item All the trees $\calT^M(-t_1, u)$, for all $u\sim v$.
\end{itemize}
The set of vertices that are either leaves of the tree or that have a subtree attached to them will be called the \emph{trunk} of the tree.
\end{itemize}

As previously, the matching tree of $\bfz$ is simply denoted by $\calT^M$, and its leaf set that is denoted by $\partial \calT^M$. 
This object in hand, we turn to the proof of the lemma.

\begin{proof}[Proof of Lemma~\ref{lem:tree_eps_perco_XY}] 
We claim the following two properties. 
\begin{itemize}
\item On the event $\{\calT^M \subset (-n, 0]\times \Lambda_n \}$, the random variables $R^{\Lambda_n, -n, 0}(\tau^{\pm})_\bfz$ are measurable with respect to: 
\begin{itemize}
\item The data of $R_{\alpha}^{\Lambda_n, -n, -t}(\tau^{\pm})_u$ for all the $(-t,u)$ belonging to the leaves set of the tree.
\item The data of the $\om$ and $\eta$-open clusters of $u$ in $R_{(\om, \eta)}^{\Lambda_n, -n, -t}(\tau^{\pm})$  for all the $(-t,u)$ belonging to the trunk of the tree.
\item The data of $\Pi(\Lambda_n \times (-n,0])$.
\end{itemize}
\item On the event that $\calT^M \subset (-n, 0] \times \Lambda_n$, that the two truncated dynamics for $\alpha^+$ and $\alpha^-$ coincide on the leaves set and the event that the $\om$ and $\eta$ open clusters of all the vertices of the trunk of the Markov tree coincide, 
then the two dynamics for $\alpha$ coincide at the vertex 0 at time 0. 
\end{itemize}

The first statement is a reformulation of the almost-Markov property, and the second consists in the observation that on the leaf set on the tree the dynamics are matched by construction, together with the first property above mentioned.  

Now let $\delta\in (0,1)$ and assume that $\eps > 0$ is sufficiently small, and $k = k(\eps)$ is chosen small enough so that:
\begin{itemize}
\item The dynamics is $(\eps, k)$-matching the digits.
\item There exists $c_\eps>0$ such that 
\begin{multline*} \P[ \{ \calT^M \subset (-\delta n, 0]\times \Lambda_{\delta n}]\} \cap \bigcap_{-r \in (-\delta n, 0] } \{R_{(\om,\eta)}^{-\Lambda_n, -n, -r}(\tau^+)|_{\Lambda_{\delta n}} = R_{(\om,\eta)}^{-\Lambda_n, -n, -r}(\tau^-)|_{\Lambda_{\delta n}}\}\\ \geq 1-\exp(-c_\eps\delta n). \end{multline*}
\end{itemize}
We briefly explain how such a choice can be done.
First, the exponential probability of coincidence of the dynamics for $(\om, \eta)$ in the space-time window $(-\delta n, 0]\times \Lambda_n$ is given by Proposition~\ref{prop:space_time_mixing_omega_eta}. 
Conditionally on this event, the percolations $(\om, \eta)$ reached their invariant measure on $\Lambda_{\delta n}$, so that they exhibit exponential decay of their clusters size by Lemma~\ref{lem:spatial_mixing_XY}.
As the trunk of the tree can be compared with a very subcritical branching process by chosing $\eps > 0$ small enough, the probability that the whole tree is not contained in $(-\delta n, 0]\times \Lambda_{\delta n}$ is exponentially small by a basic union bound on all the vertices of the tree.
The proof is now concluded exactly as in Lemma~\ref{lem:tree_eps_perco_XY}. 
\end{proof}

The reader might now check that once this adaptation is done, the rest of the proof follows as in Section~\ref{sec:spacetime} introducing 
\[ \phi(n, t) := \P[\lfloor \mathsf{NC}(n,n)  \rfloor_k  ] \]
This provides a proof of Proposition~\ref{prop:space_time_mixing_XY}.

%% file: sections_new/xy/04_conclusion_efiid.tex
\section{XY: Construction of the EFIID}
\label{sec:efiid_XY}

The argument is similar to the proof in the case of the SWM,
but the fact the Glauber dynamics in not Markovian anymore needs to be taken into account,
as a space-time surface of mixed boxes around 0 is not sufficient to decouple the state of
$R^{\infty, -\infty, 0}(\tau^{+})_0$ from the Poisson process in the $n_L  $-complement of the surface.

\subsection{Subcriticality of the set of good boxes}

Similarly as in Section~\ref{sec:efiid_SWM}, we fix a coarse-graining scale $L  \geq 0$ and assume for convenience that $\delta^{-1}L  $ is an integer.
We say that a percolation configuration is \emph{crossing} a given box if two opposite faces are linked by an open path.
The notion of mixed point needs to be replaced by the following definition.

\begin{definition}
A point $(t,x) \in L  \cdot (\Z_{\leq 0}\times\Z^d)$ is said to be $(L  ,\delta)$-good is the following three conditions are satisfied:
\begin{itemize}
\item  For all $r \in [-L  (t+1), -L  t]$, the two configurations $R^{\Lambda_{2n_L  }(v), -(t+1)n_L  , -r}(\tau^\pm )$ coincide on the space box $\Lambda_{n_L  }(v)$.
\item For all $-r \in [-L  (t+1), -L  t] $, there exists no $\om$-open cluster of $R^{\Lambda_{2n_L  }(v), -(t+1)n_L  , -r}(\tau^+ )$ crossing the space box $\Lambda_{L  }(v)$.
\item For all $-r \in [-L  (t+1), -L  t] $, there exists no $\eta$-open cluster of $R^{\Lambda_{2n_L  }(v), -(t+1)n_L  , -r}(\tau^+ )$ crossing the space box $\Lambda_{L  }(v)$.
\end{itemize}
If $(t,x)$ satisfies the first condition amongst those three, we still say that it is $(L  ,\delta)$-mixed.
\end{definition}

As previously, introduce the site percolation process on $L   \cdot(\Z_{\leq 0} \times \Z^d)$ defined by 
\[\Theta^{\delta, L  }_{(t,x)} = \ind{ \{ (t,x) \text{ is }(L  , \delta)-\text{good}\}}.\] 

\begin{lemma}\label{lem:stoch_dom_XY}
There exists a sequence $\eps_L  $ tending to 0 when $L  \rightarrow \infty$ such that the law of $\Theta^{\delta, L  }$ stochastically dominates $\P^L  _{1-\eps_L  }$, where $\P^L  _{1-\eps_L  }$ is the law of a Bernoulli bond percolation of parameter $1-\eps_L  $ on the lattice $L  \cdot (\Z_{\leq 0} \times \Z^d)$. 
\end{lemma}

\begin{proof}

The proof is the same as the corresponding lemma in Part~\ref{part:SWM}, as the conditions on $\om$ and $\eta$ are local. 
Indeed, it is still the case that for any $(t,x) \in L   \cdot(\Z_{\leq 0} \times \Z^d)$, the event $\{\Theta^{L  ,\delta}_{(t,x)} = 1\}$ is measurable with respect to
\[ \Pi(\Lambda_{2n_L  }(v), [-(t+1)n_L  , tL  ] ). \]
This implies that the process $\Theta$ is a $2\delta^{-1} + 1$ dependent process on $L  \cdot (\Z_{\leq 0} \times \Z^d)$, endowed with nearest-neighbour connectivity. 

Furthermore, by the choice of $n_L  $, Proposition~\ref{prop:space_time_mixing_XY} implies that for any $(t,x) \in L  \cdot (\Z_{\leq 0} \times \Z^d)$,
\[\P[ (t,x) \text{ is }(L  ,\delta)\text{-mixed}] \underset{L  \rightarrow 0}{\longrightarrow} 1.\]

Thus, the proof will be concluded by the result of~\cite{liggett_schonmann} provided that we prove that conditionally on fact that $(t,x)$ is $(L  ,\delta)$-mixed, 
then the probability of finding either an $\om$-open crossing cluster or an $\eta$-crossing cluster in $R^{\Lambda_{2n_L  }(v), -(t+1)n_L  , -r}(\tau^+ )$ for some $r\in [(-t+1)L  , -tL  ]$ tends to 0 as $L  $ tends to 0.
We do it for $\om$ and conclude with a union bound. 

Observe that conditionally on the fact that $(t,x)$ is $(L  ,\delta)$-mixed, $R^{\Lambda_{2n_L  }(v), -(t+1)n_L  , -r}(\tau^+ )$ follows the invariant measure $\mu$ by Theorem~\ref{thm:cftp}.
Call $\mathsf{Cross}_L  (x)$ the event that the $\om$-cluster of the point $v\in \Lambda_{L  }(x)$ is crossing.
For a given point $v$ of the box $\Lambda_{n_L  }(x)$, observe that any edge bordering $v$ is resampled $\mathsf{P}$ times in the time interval $[-L  (t+1), -L  t]$, where $\mathsf{P}$ is a Poisson variable of mean $L  $.
Thus, similarly as the proof of Lemma \ref{lem:coupling_long_time}, we may use Chernoff type upper bounds on the tail of a Poisson variable to upper bound the probability that there exists an $\om$-open cluster of $R^{\Lambda_{2n_L  }(v), -(t+1)n_L  , -r}(\tau^+ )$ crossing the box $\Lambda_L  (x)$ by
\[
 \sum_{v \in \Lambda_{n_L  }} \{100 L  \mu[\mathsf{Cross}_L  (v) ] + \exp(-cL  ) \} \underset{L   \to \infty}{\longrightarrow} 0, 
 \]
 as the $\om$-percolation in $\mu$ has exponentially small connection probabilities by Lemma~\ref{lem:spatial_mixing_XY}.
\end{proof}

The set of $(\delta,L  )$-good boxes also enjoy the same set of properties as in Part~\ref{part:SWM}. 

\begin{lemma}\label{lem:property_good_boxes_XY}
Let $(t,x) \in L  \cdot (\Z_{\leq 0} \times \Z^d)$ such that $(t,x)$ is $(L  ,\delta)$-mixed. 
Then, the following statements are true:
\begin{itemize}
\item For any $r \in [-L  (t+1), -L  t]$, the distributions of $R^{\Lambda_{2n_L  }(v), -(t+1)n_L  , -r}(\alpha^\pm)|_{\Lambda_{n_L  }(x)}$ are given by the measure $\mu^{\SWM}_{\Lambda_{2n_L  }}(\cdot|_{\Lambda_{n_L  }(x)})$.
\item For any $r \in [-L  (t+1), -L  t]$, the random variables $R^{\Lambda_{2n_L  }(x), -(t+1)n_L  , -r}(\alpha^\pm)|_{\Lambda_{n_L  }(x)}$ are independent of the Poisson process on the set $(\Lambda_{2n_L  }(x) \times [-(t+1)n_L  , tL  ] )^c$. 
\end{itemize}
\end{lemma}

\begin{proof}
Both of these statements come from the properties of the coupling from the past stated in Theorem~\ref{thm:cftp}.
\end{proof}

\subsection{Construction of the local sets for $\Theta$}

The end of the proof is extremely similar, when noticing that the two additional properties in the definition of $(\delta,L  )$-good boxes allow to decouple the inside and the outside of space-time surfaces of good boxes.

As previously, for any $L   \geq 0$, we endow the space-time set $L  \cdot(\Z_{\leq 0} \times \Z^d)$ with the usual $*$-connectivity: two vertices are considered neighbours when their $L ^\infty$ norm is equal to $L  $ in the underlying graph $\Z_{\leq 0} \times \Z^d$.
In a $\{0,1\}$-valued site percolation process on $L  \cdot(\Z_{\leq 0} \times \Z^d)$, we denote by $\calC$ the 1-cluster of the vertex 0 for the $*$-connectivity, and by $\calC^*$ the 0-cluster of the vertex 0 for the $*$-connectivity.
We start by fixing the value of $L  $ that we shall use to construct the local sets. 

\begin{lemma}\label{lem:peierls_XY}
There exists a value of $L   \geq 0$ large enough and two constants $c,C>0$ such that for any $n \geq 1$, 
\[
\P^L  _{1-\eps_L  }[|\calC^*| > n] \leq C\exp(-cn). 
\]
\end{lemma}

\begin{proof}
This follows from a standard Peierls argument for Bernoulli bond percolation of small parameter $\eps_L  $ in the graph $L  \cdot(\Z_{\leq 0} \times \Z^d)$.
\end{proof}

{\bf The value of $L  $ is now fixed and given by Lemma~\ref{lem:peierls_XY}, as well as the values of the constants $c,C$ > 0.}
We will also drop the dependency in $L  $ and in $\delta$ in the notation $\Theta^{\delta, L  }$. 

We now call $\calC^*$ the 0-cluster of the vertex (0,0) in the process $\Theta$.
The key input for constructing the local sets for $\alpha_0$ is the following observation (remember the notion of $N$-external complement that was introduce in Section~\ref{sec:efiid_SWM}).

\begin{lemma}\label{lem:decoupling_XY}
The random variable $R^{\infty, -\infty, 0}(\alpha^+)_0$ is independent of the Poisson point process $\Pi$ restricted to the set $\partial^{2n_L  }_{\mathrm{ext}} \mathcal{C}^*\times \R_{<0} $, where we recall that $n_L   = \lceil \delta^{-1}L   \rceil$. 
\end{lemma}

\begin{proof}
The lemma follows by the observation that the external boundary of $\calC^*$ consists in a space-time surface of mixed boxes shielding 0 from $\partial^{n_L}_{\mathrm{ext}}C^*$.
When updating the state of a point on the interior of that surface, the Glauber rule inspects the value of:
\begin{itemize}
\item  $\alpha$ on the neighbours of a point on the interior of the surface, and the $\eta$ and $\om$-clusters of these neighbours, wich are \emph{a priori} non-local. 
However, the second and third conditions on the definition of a $(\delta, L)$-good box prevent those connections to cross the surface. 
It means that in any case, the randomness that has been explored lies:
\begin{itemize}
\item On the interior of the surface, in which case it is independent of the Poisson point process on $\partial^{2n_L  }_{\mathrm{ext}} \mathcal{C}^*\times \R_{<0} $
\item In the surface, in which case it is independent of the Poisson point process on $\partial^{2n_L  }_{\mathrm{ext}} \mathcal{C}^*\times \R_{<0} $ by construction. 
\end{itemize}
We conclude that this update is independent of the Poisson point 
\item  Or $(\alpha, \omega, \eta)$ inside of the surface, which by assumption is independent of the Poisson point process on $\partial^{2n_L  }_{\mathrm{ext}} \mathcal{C}^*\times \R_{<0} $ by construction.
\end{itemize}  
As $(0,0)$ is contained in the space-time surface, the proof is complete.  
\end{proof}

We conclude by explaining how this property implies Theorem~\ref{thm:EFIID_SWM_XY} in the case of the XY model. 

\begin{proof}[Proof of Theorem~\ref{thm:EFIID_SWM_XY} for the XY model]

The proof is exactly the same as in Part~\ref{part:SWM}, by setting
\begin{itemize}
\item $(X_v)_{v\in\Z^d} =\Pi(\{v\}, \R_{\leq 0})_{v\in\Z^d}$,
\item $\phi(X)_v := R^{\infty, \infty, 0}(\sigma^{+1})_v$,
\item $\calL_v := \Z^d \setminus \pi (\partial^{n_L  }_{\mathrm{ext}}C^*_v)$, 
where $\pi : \Z_{\leq 0}\times\Z^d \to \Z^d, (t,x) \mapsto x$ is the spatial projection.
\end{itemize}

The reader might check that the arguments follow \emph{mutatis mutandis}.
Thus the measure $\mu^{+1}_{\XY}$ is a EFIDD.

\end{proof}

%% file: sections_new/xy/05_mixing.tex
\section{XY: Spatial mixing of the Gibbs measure}
\label{sec:XY_spatial_mixing}
\begin{proof}[Proof of Lemma~\ref{lem:spatial_mixing_XY}]
By rotating the system by an angle of $-\pi/4$, it suffices to prove that
\[
        \Big\|\mu_{\Lambda_{2n},\beta}^{e^{i\pi/4}}|_{\Lambda_n}-\mu_{\Lambda_{2n},\beta}^{e^{-i\pi/4}}|_{\Lambda_n}\Big\|_{\operatorname{TV}}
        \leq Ce^{-cn}.
\]
But in the coordinate representation,
these boundary conditions differ by flipping the sign of $\zeta$.
Thus, the above total variation estimate is bounded by
\[
    (\mu_{\Lambda_{2n},\beta}^{e^{i\pi/4}})'[\{\Lambda_n\xleftrightarrow{\eta}\partial\Lambda_{2n}\}].
\]
By the monotonicity properties, this is upper bounded by
\[
    (\mu_{\Lambda_{2n},\beta}^{i})'[\{\Lambda_n\xleftrightarrow{\eta}\partial\Lambda_{2n}\}].
\]
Just like at the beginning of Section~\ref{sec:spatial_mixing_SWM}, we now get
\begin{align}
    \Big\|\mu_{\Lambda_{2n},\beta}^{e^{i\pi/4}}|_{\Lambda_n}-\mu_{\Lambda_{2n},\beta}^{e^{-i\pi/4}}|_{\Lambda_n}\Big\|_{\operatorname{TV}}
    &
    \leq \sum_{x\in\partial\Lambda_n}(\mu_{\Lambda_{2n},\beta}^{+i})'[\{x\xleftrightarrow{\eta}\partial\Lambda_{2n}\}]
    \\
    &\asymp \sum_{x\in\partial\Lambda_n}\mu_{\Lambda_{2n},\beta}^{+1}[\sigma_x].
\end{align}
But this decays to zero exponentially fast in $n$, by the definition of $\beta_c=\beta_c(d)>\beta$.
This also implies the desired exponential decay of the connection probabilities in the statement of the lemma.
\end{proof}

%% file: sections_new/coulomb/main.tex
\input{sections_new/coulomb/01_intro.tex}

%% file: sections_new/coulomb/01_intro.tex
\section{The Villain model and the 2D Coulomb gas}
\label{sec:coulomb}

\begin{definition}[Villain model]
  \label{def:VILLAIN}
      Fix $d\in\Z_{\geq 1}$ and $\beta\in\R_{\geq 0}$.
    For any domain $\Lambda$ and $\zeta:\Z^d\to\S^1$,
    the \emph{Villain model} on $\Lambda$ with boundary condition $\zeta$
    at inverse temperature $\beta$ is the probability measure  on $\sigma\in(\S^1)^\Lambda$ given by
    \[
        \diff\mu_{\Villain,\Lambda,\beta}^\zeta(\sigma) = \frac1{Z_{\Villain,\Lambda,\beta}^\zeta} e^{-H_{\Villain,\Lambda,\beta}^\zeta(\sigma)} \diff\sigma,
    \]
    where $Z_{\Villain,\Lambda,\beta}^\zeta$ is the normalising constant,
    $\diff\sigma$ the Haar measure, and
    \begin{multline}
         e^{-H_{\Villain,\Lambda,\beta}^\zeta(\sigma)}
         =
         \Big(\prod_{\substack{\{x,y\}\subset\Lambda\\x\sim y}}p_\beta(\sigma_x,\sigma_y)\Big)
            \Big(\prod_{\substack{x\in\Lambda,\,y\in\Lambda^c\\x\sim y}}p_\beta(\sigma_x,\zeta_y)\Big);
            \\
         p_\beta(z,z'):=\sum_{m\in\Z} e^{-\frac{\beta}{2}\big(\arg(z)-\arg(z')+2\pi m\big)^2}.
    \end{multline}
    Notice that the Hamiltonian is defined somewhat atypically and incorporates
    the dependency on $\beta$; more on this later.
    The \emph{critical inverse temperature} is defined as the largest $\beta_c^{\Villain}=\beta_c^{\Villain}(d)\in(0,\infty)$
    such that $\mu_{\Villain,\Lambda_n,\beta}^{+1}[\sigma_0]$ decays exponentially fast in $n$
    for any fixed $\beta\in[0,\beta_c^{\Villain})$.
    Classical arguments based on correlations inequalities imply the existence of an infinite-volume Gibbs measure with $+1$ boundary conditions at inverse temperature $\beta\geq 0$, that we refer to as $\mu^{+1}_{\Villain, \beta}$.
    The label $\Villain$ is omitted from notations when it is clear from the context.
\end{definition}

\begin{theorem}[The Villain model is an EFIID]
  \label{thm:VILLAIN_EFIID}
  Consider the Villain model in dimension $d\in\Z_{\geq 2}$
        at inverse temperature $\beta\in[0,\beta_c^{\Villain}(d))$.
        Then the infinite-volume Gibbs measure $\mu_{\Villain,\beta}^{+1}$
        is an EFIID.
\end{theorem}

\begin{proof}
    We claim that the proof for the XY model applies to the Villain model with minor modifications.
    This holds true essentially because the Villain model may be seen as an XY model on the cable graph,
    or alternatively as a limit of XY models on graphs where the edges are replaced by multiple edges
    connected in series (this was already observed by Berezinskii).

    The only problem with the Villain model, is that if we consider the same monotone representation (on vertices and edges) as for the XY model,
    then $\omega$ and $\eta$ are no longer independent conditional on $\alpha$.
    More generally, the precise formulas for the densities must be modified,
    but this poses no real problem.

    More precisely:
    \begin{itemize}
        \item The formulas for the densities in Lemmas~\ref{lemma_density_quintuple} and
    \ref{lemma_monotonicity_angles} must be modified,
    but the qualitative monotonicity properties still hold true,
    \item The independence of $\omega$ and $\eta$ in Lemmas~\ref{lemma_density_quintuple}
    and \ref{lemma_monotonicity_edges} is lost, but the edges can still be resampled in a monotone fashion, by sampling both $\omega$ and $\eta$ from the conditional distribution,
    on one edge of the underlying graph at a time.
    \end{itemize}
    The Glauber dynamic with matching digits can still be constructed as for the XY model.
    The rest of the proof remains identical.
\end{proof}

\begin{theorem}[Analyticity of the free energy of the Villain model]
  \label{thm:VILLAIN_ANALYTICITY}
      Fix $d\in\Z_{\geq 1}$.
    Then
    \[
        \mathbf{\mathfrak{f}}_{\Villain} :\R_{\geq 0}\to\R,\,\beta\mapsto \lim_{n\to\infty} \frac1{|\Lambda_n|} \log Z_{\Villain,\Lambda_n,\beta}^{+1},
    \]
    the free energy of the model, is analytic on $[0,\beta_{c}^{\Villain}(d))$.
\end{theorem}

\begin{proof}
    We may not directly apply Theorem~\ref{thm:EFIID_implies_analyticity} because
    the Hamilontian of the Villain model does not scale linaerly with $\beta$.
    This is not really a problem for the applicability of Theorem~\ref{thm:EFIID_implies_analyticity}:
    the situation is exactly the same for the random-cluster model handled in~\cite{Ott_Analyticity_RCM}, and we refer to that work for details.
    Indeed, the only property on the Hamiltonian that we require, is the property that
    \[
        H_{\Villain,\Lambda_n,\beta+\epsilon}^0-H_{\Villain,\Lambda_n,\beta}^0
    \]
    tends to zero uniformly in $\sigma$ when $\epsilon\to 0$,
    and that the the dependence on $\epsilon$ is itself analytic.
\end{proof}